\documentclass[11pt,twoside]{preprint}
\usepackage[a4paper,innermargin=1.2in,outermargin=1.2in,
bottom=1.5in,marginparwidth=1in,marginparsep
=3mm]{geometry}
\usepackage{amsmath,amsthm,amssymb,enumerate}
\usepackage{hyperref}
\usepackage{mathrsfs}
\usepackage{breakurl}
\usepackage[nameinlink,capitalize]{cleveref}

\usepackage[osf]{newtxtext}
\usepackage[varqu,varl]{inconsolata}
\usepackage[cal=boondoxo]{mathalfa}
\usepackage[bigdelims,vvarbb,cmintegrals]{newtxmath}

\usepackage{xcolor}
\usepackage{mhequ}
\usepackage{comment}
\usepackage{microtype}
\usepackage{pifont}

\colorlet{darkblue}{blue!90!black}
\colorlet{darkred}{red!90!black}

\newtheorem{theorem}{Theorem}[section]
\newtheorem{lemma}[theorem]{Lemma}
\newtheorem{proposition}[theorem]{Proposition}
\newtheorem{corollary}[theorem]{Corollary}

\theoremstyle{definition}
\newtheorem{assumption}[theorem]{Assumption}

\theoremstyle{remark}
\newtheorem{remark}[theorem]{Remark}
\Crefname{assumption}{Assumption}{Assumptions}

\newcommand{\vn}[1]{{\vert\kern-0.23ex\vert\kern-0.23ex\vert #1 
    \vert\kern-0.23ex\vert\kern-0.23ex\vert}}
\usepackage{mathtools}

\allowdisplaybreaks

\newcommand\Z{\mathbb{Z}}

\newcommand\bone{\mathbf{1}}

\newcommand\cR{\mathcal{R}}

\newcommand\cee{\mathcal{E}}

\newcommand\cA{\mathcal{A}}

\newcommand{\cE}{\mathcal{E}}

\newcommand\cM{\mathcal{M}}

\newcommand\cP{\mathcal{P}}
\newcommand\cF{\mathcal{F}}
\newcommand\cff{\cF}

\newcommand{\D}{\partial}
\newcommand\nell{\,}

\def\eps{\varepsilon}
\def\les{\lesssim}

\newcommand{\R}{{\mathbb{R}}}
\newcommand{\Rd}{{\R^d}}

\newcommand{\E}{\mathbf{E}}
\newcommand{\bP}{\mathbf{P}}
\newcommand{\bB}{\mathbb{B}}

\newcommand{\N}{\mathbb{N}}

\def\d{\partial}

\def\({\left(}
\def\){\right)}

\DeclareMathOperator{\tr}{Tr}

\newcommand{\red}[1]{{\color{red}#1}}
\begin{document}
\title{Quantifying a convergence theorem of  Gy\"ongy and Krylov}
\author{Konstantinos Dareiotis\thanks{University of Leeds, \url{k.dareiotis@leeds.ac.uk}}, M\'at\'e Gerencs\'er\thanks{TU Wien, \url{mate.gerencser@tuwien.ac.at}}, and Khoa L\^e\thanks{TU Berlin, \url{le@math.tu-berlin.de}}}

\maketitle

\begin{abstract}
	We derive sharp strong convergence rates for the Euler--Maruyama scheme approximating multidimensional SDEs with multiplicative noise without imposing \emph{any} regularity condition on the drift coefficient.
	In case the noise is additive, we show that Sobolev regularity can be leveraged to obtain improved rate: drifts with regularity of order $\alpha \in (0,1)$  lead to rate $(1+\alpha)/2$.
\end{abstract}
\tableofcontents

\section{Introduction}\label{sec:intro}
The present article studies stochastic differential equations (SDEs) of the form
\begin{equ}\label{eq:main}
dX_t=b(X_t)\,dt+\sigma(X_t)\,dB_t,\qquad X_0=x_0,
\end{equ}
and their equidistant Euler--Maruyama approximations
\begin{equ}\label{eq:main-EM}
dX^n_t=b(X^n_{\kappa_n(t)})\,dt+\sigma(X^n_{\kappa_n(t)})\,dB_t,\qquad X^n_0=x_0,
\end{equ}
with the notation $\kappa_n(t)=\lfloor nt\rfloor/n$.
Here the initial condition is $x_0\in\R^d$; the coefficients $b:\R^d\to\R^d$ and $\sigma:\R^d\to\R^{d\times d}$ are measurable functions;  $B$ is a $d$-dimensional standard Brownian motion on a filtered probability space $(\Omega,\cF,(\cF_t)_{t\geq 0},\bP)$; and the dimension $d\in\N$ is arbitrary. 
When the coefficients are Lipschitz continuous, the convergence of $X^n$ to $X$, as well as the rate of convergence, is very well understood.
In this article we are interested in the regime where the drift $b$ is far from Lipschitz, maybe not even continuous. In fact, our first result, \cref{thm:multiplicative}, can be summarised as yielding sharp strong convergence rate without \emph{any} continuity (or even local Sobolev regularity) assumption on $b$.

Let us recall that when $b$ is irregular, the existence and uniqueness of solutions of \eqref{eq:main} relies on the regularising effects of the noise, and therefore some form of nondegeneracy of the diffusion coefficient is necessary.
Under the assumption that $\sigma$ is uniformly elliptic and sufficiently regular, the well-known result of Veretennikov \cite{Veret80} states that \eqref{eq:main} is well-posed even with merely bounded and measurable drift coefficient $b$.
However, Veretennikov's proof, using the Yamada--Watanabe principle, was not constructive and did not have any implications on the stability of the solution with respect to the Euler--Maruyama approximations $X^n$. 
It took $16$ years until the seminal work of Gy\"ongy and Krylov \cite{GyK} for the first proof that, under the same weak assumptions, $X^n$ converges in probability to $X$.
This is in stark contrast with the case of regular coefficients, where the  well-posedness and the convergence of the Euler--Maruyama approximations follow from essentially the same arguments (which are straightforward applications of the Burkholder--Davis--Gundy and Gronwall's inequalities).
The result of Gy\"ongy and Krylov was qualitative, that is, it provided no rates of convergence.
Despite significant interest (see Section \ref{sec:literature} below), there has not been \emph{any} known upper bound for the error $|X_n-X|$ in the case of bounded measurable $b$.
Our first result, \cref{thm:multiplicative} not only does provide an upper bound but also it   actually shows that this bound is of order   $n^{-1/2+\eps}$ with arbitrary $\eps>0$, which is known to be sharp even in the case of smooth coefficients (see \cite{MR1617049,MR1119837}). 

The optimality of the rate $1/2$ no longer holds when the noise is additive. If $\sigma$ is simply the identity matrix, then for smooth $b$ the rate of convergence is known to be $1$.
We are therefore interested in how much, and what kind of, regularity assumption is needed on $b$ to improve the rate of convergence.
Our second result, \cref{thm:additive} establishes rate $(1+\alpha)/2$, provided that the drift possesses Sobolev regularity of order $\alpha \in (0,1)$ with integrability exponent $p \geq \max\{2, d\}$.

The rest of the article is structured as follows. In the remainder of the introduction we briefly overview the relevant literature (Section \ref{sec:literature}), highlight the main aspects of the proof (Section \ref{sec:outline}) and state the main results of the paper (Section \ref{sec:formulation}).
In Section \ref{sec:prelim} we introduce the notation and collect/prove a number of auxiliary statements. Section \ref{sec:quad} is concerned with some quadrature estimates, which essentially provide the rate of convergence. In Section \ref{sec:pde} a general stability estimate is given for approximate solutions of \eqref{eq:main}, which is to be applied with $X$ and $X^n$.
Section \ref{sec:proof} combines all the previous ingredients to provide the proofs of the main results.

\subsection{Literature}\label{sec:literature}

The strong error analysis of the Euler--Maruyama scheme for SDEs with irregular coefficients has attracted significant attention in recent years. 
In addition to being a developing branch of the field of `regularisation by noise', it has the practical relevance that
SDEs with discontinuous drifts  have recently been utilised in  applied sciences. 
As a few examples,  they are used  in finance  for modelling equity markets (see, \cite{finance}),  in neuroscience for modelling interacting neurons that follow integrate-and-fire type dynamics (see, \cite{neurons}), and also  in the modelling of energy storage problems  (see, \cite{control}).

In one direction, see among others \cite{NT_MathComp, PT, Bao2, MX, Bao2020, Issoglo},
low regularity of order $\alpha \in (0,1)$ of the drift on the H\"older scale is assumed.
In these works the regularising properties of the noise at the level of the dual parabolic partial differential equation (PDE) are used in order to close the estimates. This is a technique which originates in the work of Zvonkin and Veretennikov \cite{Zvonya, Veret80}. 
The version that is most suitable in the setting of numerics is due to Flandoli, Gubinelli and Priola \cite{FGP}, 
and is known as the \emph{It\^o--Tanaka trick}.
This method led to various results on the rate of convergence in the case of irregular drift.
Concerning \eqref{eq:main},  in \cite{PT} the authors derived a rate  proportional to the regularity of the drift, namely, rate $\alpha/2$ for $\alpha$-H\"older continuous drift $b$,  $\alpha \in (0,1)$. Notice that this seemed to be consistent with the regular case (see, e.g., \cite[Theorem~7.5]{Pages}, where for additive noise the rate $\alpha/2$ is derived for $\alpha \in [1,2]$).
However, it turned out that this rate is suboptimal: in  \cite{DG} it is shown that in the additive noise case the rate $1/2$ is achieved even for Dini-continuous $b$, and in dimension $d=1$, for merely bounded $b$.
The main reason behind this improvement is that one can also leverage the regularising effects of the noise at a purely numerical analytic level: namely, it leads to sharp quadrature-type estimates.
For earlier works on such estimates, see \cite{Altmeyer,KHiga_quad,MR2842909}.
These bounds are closely related to the error analysis of the Euler--Maruyama scheme.

Another direction is explored in, among others, \cite{Szo1, Szo3, Y2, MY-lowerbound}, where the irregularity on $b$ is assumed to take the form of discontinuities at finitely many points (or, in higher dimensions, hypersurfaces), outside of which the usual regularity assumptions are imposed.
For a detailed account on the development here, including other approximation schemes, we refer to the introduction of
Neuenkirch and Sz\"olgyenyi \cite{NeuSz}.
Another relevant feature of \cite{NeuSz} is that, like in \cite{DG}, the regularising properties of the noise are exploited on two (PDE and quadrature) levels. With this method, in the case of dimension $d=1$ and additive noise, the authors show that for $b \in W^\alpha_2\cap L_\infty\cap L_1$, for $\alpha \in (0,1)$, the rate $\min\{ (1+\alpha)/2, 3/4\}$ can be achieved by the Euler--Maruyama scheme. Since piecewise Lipschitz functions belong (at least locally) to $W^{1/2-\eps}_2$ for any $\eps>0$, this result generalises several previous ones.
In the scalar piecewise Lipschitz case lower bounds have also been obtained by M\"uller-Gronbach and Yaroslavtseva \cite{MY-lowerbound}, showing that the rate $3/4$ is sharp.

Finally, we mention the recent work \cite{ButDarGEr}, which also works on the H\"older scale, but with quite different methods from the above. Instead of relying on PDE theory, it introduces an approach based on stochastic sewing \cite{Khoa}. This approach not only leads to improved rates (in case $b\in C^\alpha$ with additive noise, one gets rate $(1+\alpha)/2$ in \cite{ButDarGEr} vs. $1/2$ in \cite{DG} vs. $\alpha/2$ in \cite{PT}), but also widely extends the scope of driving noise, covering non-Markovian examples like fractional Brownian motions.

The contributions of the present article in relation to the existing results are as follows.
\begin{itemize}
\item In the multiplicative noise case \cref{thm:multiplicative} provides the first, and at the same time, sharp, quantification of the qualitative theorem of Gy\"ongy and Krylov \cite{GyK}. Previous results imposed, in one way or another, positive regularity (e.g., H\"older \cite{ButDarGEr} or -- implicitly -- Sobolev \cite{Szo3}).
The case of $0$ regularity is critical from an analytic point of view, as demonstrated by the fact that up to now not even suboptimal rates were available in this borderline case. 
\item In the additive noise case \cref{thm:additive} shows that if the drift does have additional regularity on the Sobolev scale, then the rate of convergence is also improved.
In comparison to \cite{ButDarGEr}, regularity is assumed on the Sobolev, rather than the H\"older scale, allowing discontinuous coefficients. 
Even when $b\in W^{\alpha}_p$ with $p>d/\alpha$ and therefore one has by Sobolev embedding $b\in C^{\alpha-d/p}$, the present results can not be recovered from \cite{ButDarGEr}. Indeed, \cite{ButDarGEr} would imply rate $(1+\alpha-d/p)/2$ in this case, while \cref{thm:additive} shows rate $(1+\alpha)/2$.
In comparison to \cite{NeuSz} which considers drifts from Sobolev spaces, our methods give stronger results in two main directions. First, for drifts with Sobolev regularity $\alpha$ the Euler--Maruyama scheme is shown to have rate $(1+\alpha)/2$, which removes the $3/4$ threshold from the result of \cite{NeuSz}.
Second, \cref{thm:additive} is valid in all dimensions, in contrast to the restriction $d=1$ from \cite{NeuSz}. Other improvements include higher moment estimates (provided $b$ has Sobolev regularity with higher integrability exponent) and uniform in time error estimates.
\end{itemize}

\subsection{On the proof}\label{sec:outline}
We now briefly outline the strategy of the proof, also highlighting the differences/similarities to previous works. 
For simplicity let us consider the additive case, with constant identity diffusion matrix. One has the standard error decomposition
\begin{equ}\label{eq:decomposition}
X_t-X^n_t=\int_0^t\big(b(X_s)-b(X^n_s)\big)\,ds+\int_0^t \big(b(X^n_s)-b(X^n_{\kappa_n(s)})\big)\,ds.
\end{equ}
To have any chance of obtaining a Gronwall-type estimate, one would like to bound the first integral by $\|X-X^n\|$, in some norm $\|\cdot\|$. In all but one of the above mentioned works this is achieved by means of PDE techniques, using what is known as the It\^o--Tanaka trick from \cite{FGP}, which in turn is a variation of the methods of \cite{Zvonya, Veret80}.
Recently, in \cite{ButDarGEr} a new approach was introduced to obtain such bounds, based on the sewing methods inspired by \cite{Cat-Gub, Khoa}.
Coincidentally, the sewing method also turns out to be efficient in handling the second integral in \eqref{eq:decomposition}.
For example, if one takes the coefficient $b$ from the H\"older space $C^{\alpha}$, $\alpha\in(0,1)$, then instead of the naive bound of order $\|b\|_{C^\alpha}n^{-\alpha/2}$
one can obtain a bound of order $\|b\|_{C^\alpha}n^{-(1+\alpha)/2}$, see \cite[Lemma~4.2]{ButDarGEr}.

The lack of regularity of $b$ poses difficulties in both steps.
Therefore, in the present article we follow a hybrid path: the first term in \eqref{eq:decomposition} is treated by the PDE approach and the second one is estimated by stochastic sewing.
Concerning the PDE step, in the Zvonkin/It\^o--Tanaka transformation
the solution of the corresponding PDE  (see \eqref{eq:PDE}) does not have bounded second derivatives, which makes closing the estimates with Gronwall-type argument problematic.
This is resolved by relying instead on local $L_p$-bounds on the second derivatives, replacing Lipschitz bounds with the Hardy--Littlewood maximal inequality, and finally controlling the localisation error.
While localisation steps are quite common in the literature when it comes to uniqueness of SDEs (usually appearing as a simple  stopping times argument), they have to be made quantitative for the error analysis.
Let us also mention that the use of Hardy--Littlewood maximal inequality in a similar context (albeit without localisation) also appeared in the recent work \cite{Bao2020}.
Concerning the stochastic sewing step, the main novelty compared to \cite{ButDarGEr} is to exploit Sobolev regularity in estimating approximation errors for `occupation time functionals', that is, bounding quantities of the form
\begin{equation*}
\int_0^t \big(f(B_s)-f(B_{\kappa_n(s)})\big)\,ds,
\end{equation*}
with $f$ belonging to certain Sobolev space, see Lemma \ref{lem:(ii)}. This requires a version of the stochastic sewing lemma of \cite{Khoa} allowing singularities, by introducing temporal weights, see Lemma \ref{lem:wSSL}.

The hybrid approach described above seems useful and efficient. In particular, it has been applied and extended in \cite{le2021taming} to obtain sharp rates of convergence of tamed Euler schemes for SDEs with integrable drifts, as considered in Krylov--R\"ockner \cite{MR2117951}.
\subsection{Main results}\label{sec:formulation}
\begin{assumption}\label{asn:multiplicative-basic}
Assume that $b$ is bounded, $\sigma$ is twice differentiable, its derivatives of order $0$, $1$, and $2$ are bounded functions, and that for some $\lambda>0$, one has $|y^* (\sigma \sigma^*)(x) y|\geq \lambda^2|y|^2$ for all $x,y\in\R^d$.
\end{assumption}
Recall that $X$ and $X^n$ are defined as the solutions of \eqref{eq:main} and \eqref{eq:main-EM}, respectively.
The ellipticity of $\sigma$ is assumed in order to ensure solvability when $b$ is irregular.
The fact that under \cref{asn:multiplicative-basic} (and therefore also under the stronger \cref{asn:additive-basic} below) these solutions exist and are unique, follows immediately from \cite{Veret80}.
The nondegeneracy of the noise can also be highly relevant in the context of approximations: as shown \cite{Hairer2015}, even if all coefficients are globally bounded with locally bounded derivatives of any order, with degenerate diffusion coefficient the strong rate of convergence may be worse than any positive power.
Our results also imply that such behaviour can be excluded for elliptic diffusion, and see the discussion in Section \ref{sec:local} below how the results can be extended to a certain class of degeneracies.
The first main result of the article is the following.
\begin{theorem}\label{thm:multiplicative}
Given  \cref{asn:multiplicative-basic}, let $\eps\in(0,1)$, $p\in(0,\infty)$.
Then for all $n\in \N$ the following bound holds
\begin{equ}\label{eq:main-bound-multiplicative}
\big(\E\sup_{t\in[0,1]}|X_t-X^n_t|^p\big)^{1/p}\leq N n^{-1/2+\eps},
\end{equ}
with some constant $N$ depending only on $d,p,\eps,\lambda$ and $\sup\big(|b|+|\sigma|+|\nabla\sigma|+|\nabla^2\sigma|\big)$.
\end{theorem}

In the additive case, an even higher rate can be shown given some Sobolev regularity of the drift $b$.  The homogeneous Sobolev spaces $\dot{W}^\alpha_p(\R^d)$,  for $\alpha\in(0,1)$ and $p\in[1,\infty)$, are defined as usual: they contain all measurable functions $f : \R^d \to \R$ such that 
\begin{equ}\label{eq:Sobolev}
\,[f]_{\dot{W}^{\alpha}_{p}(\R^d)}:=\bigg(\int_{\R^d}\int_{\R^d}\frac{|f(x)-f(y)|^p}{|x-y|^{d+\alpha p}}\,dx\,dy\bigg)^{1/p}< \infty. 
\end{equ}
This definition obviously extends to finite dimensional vector-valued functions.

\begin{remark}   \label{rem:choice-of-rep}
If $f \in \dot{W}^{\alpha}_{m}(\R^d)$ and $\alpha>d/m$ then $f$ has a version that is continuous and $[ f]_{C^{\alpha-d/m}(\R^d)} \le N(d, \alpha, m) [f]_{\dot{W}^\alpha_m(\R^d)}$ (see, e.g., \cite[Lemma 2 and Remark 3, p. 203-206]{Krylov_PDE}).  If $\alpha>d/m$,   then  the elements of  $\dot{W}^{\alpha}_{m}(\R^d)$  will be  treated as continuous functions rather than equivalence classes.  In particular, if we write that  $b \in \dot W^\alpha_m(\R^d)$ (as in Assumption \eqref{asn:additive-basic}) and $\alpha >d/m$,  then we automatically mean that, in  addition,  $b \in C^{\alpha-d/m}$. 
\end{remark}

\begin{assumption}\label{asn:additive-basic}
Assume that $\sigma=I$ and $b$ is a bounded measurable function which belongs to $\dot{W}^{\alpha}_{m}(\R^d)$ for some $\alpha\in (0,1)$ and $m\geq\max(d,2)$.
\end{assumption}

\begin{theorem}\label{thm:additive}
Let \cref{asn:additive-basic} hold, let $\eps\in(0,1)$, $p\in(0,m]$.
Then for all $n\in \N$ the following bound holds
\begin{equ}\label{eq:main-bound-additive}
\big(\E\sup_{t\in[0,1]}|X_t-X^n_t|^p\big)^{1/p}\leq N n^{-(1+\alpha)/2+\eps},
\end{equ}
with some constant $N$ depending only on $d,p,\eps,\alpha,m,\sup|b|$ and $[b]_{\dot{W}^\alpha_m(\R^d)}$.
\end{theorem}

\begin{remark}
In the rest of the paper we will actually assume strict inequality in the conditions  $m\geq\max(d,2)$, and $m\geq p$. To see that this is not a restriction,   by Sobolev embeddings
we can always slightly increase $m$ at the price of slightly decreasing $\alpha$.
In the situation of \cref{asn:additive-basic} even a better embedding is available, namely that $b$ belongs to $\dot{W}^{\theta\alpha}_{m/\theta}$ for any $\theta\in(0,1)$, see \cref{lem.triv.em}.
This means a slight loss in the rate, which can just be absorbed in the $\eps$.
\end{remark}

\begin{remark}
Note that there is no moment restriction in \cref{thm:multiplicative} but there is the restriction $p\leq m$ in \cref{thm:additive}. One can increase the moments by sacrificing from the rates, using again the embedding from Lemma \ref{lem.triv.em}.
\end{remark}

As mentioned before, one interesting class of coefficients with Sobolev regularity is that of indicator functions of regular domains. Indeed, it is well-known (see, e.g., \cite[Section~3.2]{Indicators}) that if $D\subset\R^d$ is a bounded Lipschitz domain, then $\bone_D\in \dot{W}^{1/p-\eps}_p$ for every $p\in[1,\infty)$, $\eps>0$. Since multiplying with a bounded Lipschitz function leaves $\dot{W}^{1/p-\eps}_p$ invariant, one gets the following corollary.

\begin{corollary}\label{cor:indicators}
Let $\sigma=I$, $\eps\in(0,1)$, and $p\in(0,\infty)$.
Assume that with a finite set of bounded Lipschitz domains $D_1,\ldots, D_k$ and bounded Lipschitz continuous functions $f_1,\ldots,f_k$, $b$ is of the form
\begin{equ}
b(x)=\sum_{i=1}^kf_i(x)\bone_{D_i}(x).
\end{equ}
Then for all $n\in\N$ the following bound holds
\begin{equ}
\big(\E\sup_{t\in[0,1]}|X_t-X^n_t|^p\big)^{1/p}\leq N n^{-\frac{1}{2}\big(1+\frac{1}{\max(2,d,p)}\big)+\eps},
\end{equ}
with some constant $N$.
\end{corollary}
\begin{proof}
By the preceding remarks, $b$ belongs to $\dot{W}^{1/m- \varepsilon'}_m$ for $m=\max(2,d,p)$ and for every $\varepsilon'>0$, hence satisfying  \cref{asn:additive-basic}.
The result follows immediately from \cref{thm:additive}.
\end{proof}

\begin{remark}\label{rem:indicators}
In the special case of $d=1$, Corollary \ref{cor:indicators} yields the $L_2$-rate $3/4-\eps$. As mentioned above, this is known to be sharp, see  \cite{MY-lowerbound}.
We do not know whether in general dimensions $d>1$ the $L_2$-rate $(d+1)/(2d)$ is sharp, but it is certainly the best known bound at the moment.
\end{remark}

\subsection{Localising irregularities}\label{sec:local}

Let us now briefly outline how to some extent the results can be extended to coefficients with degeneracy and/or growth, provided these properties do not ``interfere'' with the irregularities.

Let $A^1,A^2$ be two Lipschitz domains such that $\bar A^1\cup \bar A^2=\R^d$.
Take some $\delta>0$ and denote the $\delta$-fattening of $A^i$ by $A^i_\delta$ for $i=1,2$.
We assume that there exist (globally defined) coefficients $b^i,\sigma^i$ such that $b=b^i$ and $\sigma=\sigma^i$ on $A^i_{3\delta}$, and that furthermore the corresponding SDEs and their Euler approximations satisfy the strong Markov property and for all $n\in\N$ the bound
\begin{equ}
\E\sup_{t\in[0,1]}| X_t^i- X^{i,n}_t|^p\leq K n^{-\alpha}
\end{equ}
holds for some $K<\infty$, $p\in[1,\infty)$, and $\alpha>0$ uniformly over initial conditions.
Furthermore, assume the a priori estimates
\begin{equ}
\E\sup_{t\in[0,1]}|X_t|^q+\sup_{n\in\N}\E\|X^n\|_{C^\kappa[0,1]}^q\leq K(q)
\end{equ}
for some $\kappa\in(0,1]$ and for all $q\in[1,\infty)$.
One may think of an example where $b$ is bounded and $\sigma$ is elliptic on a ball and both coefficients are Lipschitz but possibly degenerate and have some (linear) growth on the complement of a smaller ball (and therefore ``best of both worlds'' on the intersection).
We then claim that
\begin{equ}
\E\sup_{t\in[0,1]}| X_t- X^{n}_t|^p\leq Nn^{-\alpha+\eps}
\end{equ}
for any $\eps>0$, where $N=N(K, (K(q))_{q\in[1,\infty)},p,\eps,\delta)$.
Without loss of generality we assume $x_0\in \bar A^1$, and we define the stopping times $\tau_0=0$, and for $k=0,1,\ldots$,
\begin{equ}
\tau_{2k+1}=\inf\{t\geq \tau_{2k}: tn\in\N, X^n_t\notin A^1_\delta\},\quad
\tau_{2k+2}=\inf\{t\geq \tau_{2k+1}: tn\in\N, X^n_t\notin A^2_\delta\}.
\end{equ}
For any $q\in(1,\infty)$, with dual exponent $q'=q/(q-1)$ we have by the assumptions and standard use of H\"older's and Markov's inequalities
\begin{equs}
\E&\sup_{t\in[0,\tau_1]}|X_t-X^n_t|^p\lesssim \E\sup_{t\in[0,\tau_1]}|X_t^1-X^{1,n}_t|^p+\big(\bP(X^{1} \text{ exits }A^1_{3\delta}\text{ before }\tau_1) \big)^{1/q'}
\\
&\lesssim n^{-\alpha}+\big(\bP(X^{n,1} \text{ exits }A^1_{2\delta}\text{ before }\tau_1)\big)^{1/q'} + \big(\bP(\sup_{t\in[0,\tau_1]}|X_t^1-X^{1,n}_t|\geq\delta)\big)^{1/q'}
\\
&\lesssim n^{-\alpha} +\big(\bP(\|X^n\|_{C^{\kappa}[0,1]}\geq \delta n^\kappa)\big)^{1/q'}+ n^{-\alpha/q'}
\\
&\lesssim n^{-\alpha/q'}+n^{-\kappa q/q'}.
\end{equs}
The proportionality constants in $\lesssim$ always depend only on the parameters of $N$ mentioned above and $q$.
Choosing $q$ large enough we get a bound of order $n^{-\alpha+\eps}$.
Similarly, by the strong Markov property, we have for any $k\geq 0$ that
\begin{equ}
\E\sup_{t\in[\tau_k,\tau_{k+1}]}|X_t-X^n_t|^p\lesssim n^{-\alpha+\eps}.
\end{equ}
Therefore, for any $m\in\N$ and $q\in(1,\infty)$,
\begin{align*}
\E\sup_{t\in[0,1]}|X_t-X_t^n|^p 
&\lesssim \sum_{k=0}^{m-1}\E\sup_{t\in[\tau_k,\tau_{k+1}]}|X_t-X^n_t|^p+\E\big(\sup_{t\in[0,1]}|X_t-X_t^n|^p \bone_{(\tau_m<1)}\big)
\\&\lesssim m n^{-\alpha+\eps}+\big(\bP(\tau_m<1)\big)^{1/q'}
\\&\lesssim m n^{-\alpha+\eps}+\big(\bP(\|X^n\|_{C^\kappa[0,1]}\geq\delta m^\kappa)\big)^{1/q'}
\\
&\lesssim m n^{-\alpha+\eps}+m^{-\kappa q/q'}.
\end{align*}
Choosing $m=n^\eps$ and then $q$ large enough, we get the claim with $2\eps$ in place of $\eps.$

\section{Preliminaries}\label{sec:prelim}
\subsection{Notation}
\paragraph*{Function spaces}
For any function $f\colon Q\to V$, where $Q\subset \R^k$ is a Borel set and $(V,|\cdot|)$ is a  normed space, with the notation $\N_0:= \N \cup\{0\}$,  let us set the (semi-)norms
\begin{equs}
\|f\|_{C^0(Q,V)}&=\sup_{x\in Q}|f(x)|\,;& &\\
\,[f]_{C^\gamma(Q,V)}&=\sum_{\substack{\ell\in\N_0^k\\|\ell|_1=\hat\gamma}}\sup_{x\neq y\in Q}\frac{|\d^{\ell}f(x)-\d^{\ell}f(y)|}{|x-y|^{\bar\gamma}}\,,&\qquad&\gamma>0,\,\gamma=\hat\gamma+\bar\gamma,\,\hat\gamma\in\N_0,\bar\gamma\in(0,1];\\
\|f\|_{C^\gamma(Q,V)}&=\sum_{\substack{\ell\in\N_0^k\\|\ell|_1< \gamma}} \|\D^\ell f\|_{C^0(Q,V)}+[f]_{C^\gamma(Q,V)},&\qquad &\gamma>0.
\end{equs}
In the above $|\ell|_1=\ell_1+\ldots+\ell_k$ if $\ell=(\ell_1,\ldots,\ell_k)\in \N_0^k$.
By $C^\gamma(Q,V)$ we denote the space of all measurable functions $f : Q \to V$ such that $\|f\|_{C^\gamma(Q,V)}<\infty$.

We will also denote $\bB=C^0$,  emphasizing
that elements of $C^0$ need not be continuous, but only bounded measurable (and avoiding the confusion with the space of bounded continuous functions). In addition, notice that the elements of $\bB$ are function rather than equivalent classes. 
Similarly, for $k \in \N$, $C^k$ functions are $(k-1)$-times continuously differentiable and the derivatives of $(k-1)$-order are Lipschitz continuous. For $\gamma \in \R_+ \setminus \N_0$,  $C^\gamma$ are  of course the usual H\"older spaces. 
When the domain $Q$ is $\R^d$ and/or the target space $V$ is $\R^d$ or $\R^{d\times d}$, they are suppressed from the notation.
Moreover, in $\R^d$ or $\R^{d\times d}$, $|\cdot|$ is understood to be the Euclidean norm.
For $\gamma<0$, we denote by $C^\gamma$ the space of all tempered distributions $f $ such that 
\begin{equs}
\|f\|_{C^\gamma} := \sup_{t \in (0,1)} t^{-\gamma/2} \| \cP_t f\|_{L_\infty(\R^d)} < \infty, 
\end{equs}
where $(\mathcal{P}_t) _{t \geq 0}$ is the heat semigroup associated to the standard Gaussian kernel $p_t(x)=(2 \pi t)^{-d/2}e^{-|x|^2/(2t)}$.  Notice that for $\gamma_1, \gamma_2 \in  \mathbb{R}$ with $\gamma_1 < \gamma_2$ we have the continuous embedding $C^{\gamma_2} \hookrightarrow C^{\gamma_1}$. Recall the definition of $\dot W^{\alpha}_p$ from \eqref{eq:Sobolev}.
 The following simple property is used in some of the remarks in the introduction.
\begin{lemma}\label{lem.triv.em}
	For $\alpha, \theta\in(0,1)$, $m\in[1,\infty)$, one has the inclusion 
	$\bB\cap \dot{W}^\alpha_m\subset \bB\cap \dot{W}^{\alpha \theta}_{m/\theta}$. 
\end{lemma}
\begin{proof}
	Let $f$ be a bounded function in $\dot{W}^\alpha_p$. We have
	\begin{align*}
		[f]_{\dot{W}^{\alpha \theta}_{m/\theta}(\R^d)}
		&=\big(\int_{\R^d}\int_{\R^d}\frac{|f(x)-f(y)|^{m/ \theta}}{|x-y|^{d+\alpha m}}\,dx\,dy\big)^{\theta/m},
		\\&\le2\|f\|_{\bB}^{1- \theta} \big(\int_{\R^d}\int_{\R^d}\frac{|f(x)-f(y)|^{m}}{|x-y|^{d+\alpha m}}\,dx\,dy\big)^{\theta/m},
	\end{align*}
	which means that $[f]_{\dot{W}^{\alpha \theta}_{m/\theta}(\R^d)}\le 2\|f\|_{\bB}^{1- \theta}[f]_{\dot{W}^{\alpha}_{m}(\R^d)}^\theta$, completing the proof. 
\end{proof}

\paragraph*{Matrices} We use the following conventions and basic properties of $d\times d$ matrices. By $I$ we denote the identity matrix and by $A^\ast$ the transpose of $A$. The operator norm, determinant and trace of $A$ are denoted respectively by $\|A\|$, $\det(A)$ and $\tr(A)$. For two symmetric matrices $A_1,A_2$, by $A_1\prec A_2$ we mean that $A_2-A_1$ is non-negative definite. 
Recall that the determinant is a differentiable function in a neighbourhood of the identity matrix, and therefore there exists a constant $N=N(d)$ such that $|1-\det(I+A)|\leq N \|A\|$ for $\|A\|\leq 1/2$. 
As a simple consequence, for any fixed $K$, there exists a constant $N=N(d,K)$ such that on the set $\{A:\|A\|\leq K\}$, one has 
\begin{equ}\label{eq:simple-matrix-bound}
|1-|\det (I+A)|^{1/2}|\leq N \|A\|.
\end{equ}
Indeed, by the previously mentioned bound 
there exists a $\delta=\delta(d)$ such that $|1- \text{det}(I+A)| \leq 1/2$ for $\| A \| \leq \delta$.  Since the square root function is Lipschitz on $[1/2, 3/2]$ (even with Lipschitz constant $1$),  for all $A$ with $\| A\| \leq \delta$ we have 
\begin{equs}
|1- |\text{det}(I+A)|^{1/2} |  \leq  |1- \text{det}(I+A)| \leq N(d) \| A\|.
\end{equs}
This is the claimed bound for $\| A\| \leq \delta$.
On the other hand,  if $K\geq \delta$ and  $ \delta \leq \| A\| \leq K$, then trivially we have 
\begin{equs}
|1- |\text{det}(I+A)|^{1/2} |   \leq 1+ N'(d) \| I+A\|^{d/2} \leq N''(d, \delta, K) \| A \|.
\end{equs}
\paragraph*{Convention on constants} In proofs (and only in proofs) of theorems/lemmas/propositions we use the shorthand $f\lesssim g$ to mean that there exists a constant $N$ such that $f\leq N g$, and that $N$ does \emph{not} depend on any other parameters than the ones specified in the theorem/lemma/proposition. Whenever a constant depends on any other parameters, they are indicated by parentheses in $N(\cdot)$.

\subsection{Heat kernel bounds}\label{sec:HK}
For a symmetric positive definite $d \times d$ matrix $\Sigma$, let $p_\Sigma$ be the density of a $d$-dimensional Gaussian vector with mean zero and covariance matrix $\Sigma$ defined by
\begin{equ}\label{eq:p-def}
p_\Sigma(x)=\frac{(\det\Sigma^{-1})^{1/2}}{(2\pi)^{d/2}}\exp\big(-\tfrac{1}{2}x^*\Sigma^{-1}x\big),\quad x\in\R^d.
\end{equ}
Let $\Sigma^{1/2}$ be a square matrix such that $\Sigma=\Sigma^{1/2}(\Sigma^{1/2})^*$ and introduce the notation $x_\Sigma=\Sigma^{-1/2}x$ where $\Sigma^{-1/2}=(\Sigma^{1/2})^{-1}$. The exponential in \eqref{eq:p-def} can be rewritten as $\exp(-\tfrac{1}{2}|x_\Sigma|^2)$. It is then straightforward to see that
for any $k\geq 0$,
\begin{equ}\label{eq:easy-gaussian}
|x_\Sigma|^kp_{\Sigma}(x)\leq Np_{\Sigma/2}(x),
\end{equ}
where $N$ depends only on $k$ and $d$.
For $t>0$, we use the shorthand $p_t=p_{tI}$.
For a measurable function $f\colon\R^d\to\R$ we write $\cP_\Sigma f:=p_\Sigma\ast f$ and $\cP_t f:=p_t\ast f$. For all $t>0$, $\theta\in[1,\infty]$, and $\alpha\geq 0$ one has the bounds
\begin{gather}
\|p_t\|_{L_\theta(\R^d)}\leq N t^{-\frac d2(1-\frac1\theta)}, \label{eq:HK-Lp-1}
\\
\|\nabla p_t\|_{L_\theta(\R^d)}\leq N t^{-\frac12-\frac d2(1-\frac1\theta)}, \label{eq:HK-Lp-1-and-a-half}
\\
\big\||\cdot|^\alpha p_t(\cdot)\big\|_{L_\theta(\R^d)}\leq N' t^{\frac\alpha2-\frac d2(1-\frac1\theta)},\label{eq:HK-Lp-2}
\\
\big\||\cdot|^\alpha \nabla^2  p_t(\cdot)\big\|_{L_\theta(\R^d)}\leq N' t^{\frac\alpha2-1-\frac d2(1-\frac1\theta)},\label{eq:HK-Lp-3}
\end{gather}
with some constants $N=N(d, \theta)$ and $N'=N'(\alpha, d, \theta)$.

\begin{lemma}      \label{lem:hessian-seminorm}
Let $p \in [1, \infty)$ and $\alpha \in (0,1)$. Then, for all $f \in \dot{W}^\alpha_p(\R^d)$, $t\in(0,1]$,
\begin{equation*}
\big\| \nabla^2  \cP_t f \big\|_{L_p(\R^d)} \leq N t^{-1+\frac\alpha2} [ f ]_{\dot{W}^\alpha_p(\R^d)},
\end{equation*}
where $N$ depends only on $\alpha, p$ and $d$. 
\end{lemma}

\begin{proof}
For $p,q \in [1, \infty)$, $\frac1q+\frac1p=1$, by H\"older's inequality, we have for any $\beta >0$ 
\begin{equs}
\big\| \nabla^2  \cP_t f \big\|_{L_p(\R^d)}^p & = \int_{\R^d} \Big| \int_{\R^d} ( \nabla^2 p_t)(y) f(x-y) \, dy \Big|^p \, dx 
\\
& = \int_{\R^d} \Big| \int_{\R^d} ( \nabla^2 p_t)(y) \Big(  f(x-y) -f(x) \Big) \, dy \Big|^p \, dx 
\\
& \leq \big\| | \cdot |^\beta \nabla^2 p_t(\cdot) \big\|_{L_q(\R^d)} ^p \int_{\R^d} \int_{\R^d} \frac{|f(x-y)-f(x)|^p}{|y|^{p\beta}} \, dx dy.
\end{equs}
Choosing $\beta = d/p+ \alpha$ gives 
\begin{equs}
\big\| \nabla^2  \cP_t f \big\|_{L_p(\R^d)}^p \leq  \|  | \cdot |^\beta \nabla^2 p_t(\cdot)  \|_{L_q(\R^d)} ^p[ f ]^p _{\dot{W}^\alpha_p(\R^d)},
\end{equs}
and the claim follows by \eqref{eq:HK-Lp-3}. The case $p = \infty$ follows similarly. 
\end{proof}

\begin{lemma}     \label{lem:old-heat-estimate}
Let $ \beta\in \R $,   $\alpha\geq (-\beta) \vee 0$. There exists a constant $N$ depending only on $\alpha, \beta$ and $d$ such that for all $t \in (0,1]$ and for all $f\in C^\beta(\Rd)$,
$$
\|\cP_tf\|_{C^{\beta+\alpha} (\R^d)}\le N t^{-\alpha/2} \|f\|_{C^\beta(\R^d)}. 
$$
\end{lemma}
\begin{proof}
For $\beta \geq 0$ these are well known estimates that follow from direct computations with Gaussian densities.
If $\beta< 0$, then the case $ \alpha +\beta \in [0,1]$ is shown in \cite[Proposition 3.7]{ButDarGEr}. The case $\alpha+ \beta>1$ now follows by using the semi-group property. If $\alpha+ \beta \in (1, 2]$,  for example, then 
\begin{equs}
\|\cP_tf\|_{C^{\beta+\alpha} (\R^d)} & \leq \| \nabla \cP_{\frac{t}{2}+\frac{t}{2}} f\|_{C^{\beta+\alpha-1} (\R^d)}  + \|\cP_tf \|_{\bB(\R^d)}
\\
& \lesssim t^{ -(\beta+\alpha-1)/2}\| \cP_{\frac{t}{2}} f\|_{C^1 (\R^d)} +t^{\beta/2}\|f\|_{C^\beta  (\R^d)} 
\\
& \lesssim  t^{ -(\beta+\alpha-1)/2} t^{(\beta-1)/2} \|f\|_{C^\beta (\R^d)} + t^{\beta/2}\|f\|_{C^\beta  (\R^d)} 
\\
& \lesssim  t^{-\alpha/2} \|f\|_{C^\beta(\R^d)}.
\end{equs} 
The case $\alpha+\beta \in (k, k+1]$, for $k>1$, follows by induction. 
\end{proof}
Let $\Delta$ denote the Laplacian on $\Rd$. The next lemma is folklore, but since we did not find an exact reference in this form, we provide a short proof.
\begin{lemma}\label{lem:schauder}
Let $\beta \in \R \setminus \Z $. There exists a constant $N$ depending on $\beta$ such that for all $f \in C^\beta(\Rd)$
\begin{equation*}
\| (1-\Delta)^{-1} f  \|_{C^{\beta+2}(\Rd)} \leq N \| f\|_{C^\beta(\Rd)}. 
\end{equation*}
\end{lemma}

\begin{proof}
If $\beta \in\R \setminus \Z$ and $\beta>0$, these are the usual Schauder estimates. Hence, we only deal with the case $\beta< 0$. First, let us look at the case $-1 < \beta <0$.  Recall the following simple properties of $\cP_t$: it commutes with differentiation $\nabla \cP_t f=\cP_t\nabla f$, it satisfies the semigroup property $\cP_t\cP_s=\cP_{t+s}$, and one has
$$
(1-\Delta)^{-1}f = \int_0^\infty e^{-t}  \cP_t f  \, dt.
$$
Combining with  Lemma \ref{lem:old-heat-estimate} gives 
\begin{equs}
\| (1-\Delta)^{-1}f \|_{C^1 (\R^d)} \lesssim  \int_0^1  t^{(\beta-1)/2}\|  f \|_{C^\beta (\R^d)} \, dt +  \sup_{t \geq 1} \| \cP_t f \|_{C_1(\R^d)}. 
\end{equs}
For the second term on the right hand side we have the following. Suppose that $\gamma \geq 0$. Then, for all $t \geq 1$, we have 
\begin{equs}
\|\cP_tf \|_{C^\gamma(\R^d)} \leq \|\cP_1f \|_{C^\gamma(\R^d)} \lesssim  \|\cP_{1/2}f \|_{L_\infty (\R^d)}  \lesssim \| f \|_{C^\beta(\R^d)},
\end{equs}
where for the second inequality we used the semi-group property and Lemma \ref{lem:old-heat-estimate} and for the last one the definition of $C^\beta$. 
Consequently, 
$\| (1-\Delta)^{-1}f \|_{C^1 (\R^d)} \leq N  \| f \|_{C^\beta(\R^d)}$ and we only have to show (recall that $\beta+1 \in (0,1)$) that 
$$
[ \nabla (1-\Delta)^{-1}f ]_{C^{\beta+1} (\R^d)} \lesssim  \| f \|_{C^\beta(\R^d)}.
$$
It is known that (see, e.g., \cite[Exercise 3.4.4, p.39 ]{KrylovHolder} )
\begin{align*}
[ \nabla(1-\Delta)^{-1}f ] _{C^{\beta+1} (\R^d)} \lesssim \sup_{ \eps >0} \eps^{-\beta/2}\|\nabla  \cP_\eps  \nabla (1-\Delta)^{-1}f \|_{L_\infty(\R^d)}.
\end{align*}
Let $\eps  \in (0,1]$. By the above mentioned properties of $\cP$ and by Lemma \ref{lem:old-heat-estimate} we have
\begin{equs}
\|\nabla \cP_\eps  \nabla (1-\Delta)^{-1}f \|_{L_\infty(\R^d)} & \leq \int_0^ 1\| \nabla  \cP_{\frac{t +\eps}{2}}   \nabla\cP_{\frac{t +\eps}{2}} f \|_{L_\infty(\R^d)}   \, dt +\int_1^\infty e^{-t} \|  \nabla^2  \cP_{t +\eps}    f \|_{L_\infty(\R^d)} \, dt
\\ 
& \lesssim  \int_0^1 (t+\eps)^{-1/2} \|    \nabla \cP_{\frac{t +\eps}{2}} f \|_{L_\infty(\R^d)} \, dt +  \|   f \|_{C^\beta(\R^d)}
\\
& \leq  \int_0^ \infty (t+\eps)^{-1/2} \|   \nabla \cP_{\frac{t +\eps}{2}} f \|_{L_\infty(\R^d)} \, dt +   \|   f \|_{C^\beta(\R^d)}
\\
& \lesssim  \int_0^ \infty (t+\eps)^{-1/2} (t+\eps)^{(\beta-1)/2}\|   f \|_{C^\beta(\R^d)}  \, dt +  \|   f \|_{C^\beta(\R^d)}
\\
& \lesssim \eps^{\beta/2} \|   f \|_{C^\beta(\R^d)}.
\end{equs}
We now show that the same estimate holds for $\eps >1$. From \eqref{eq:HK-Lp-1-and-a-half} we have $\|\nabla \cP_t g \|_{L_\infty}\le\|\nabla p_t\|_{L_1}\|g\|_{L_\infty} \lesssim t^{-1/2} \|g\|_{L_\infty}$ for all $t>0$. Consequently, for $\eps>1$, we have 
\begin{equs}
\|\nabla \cP_\eps  \nabla(1-\Delta)^{-1}f \|_{L_\infty(\R^d)} & \leq \int_0^ \infty e^{-t}\| \nabla  \cP_{\frac{t +\eps}{2}}   \nabla \cP_{\frac{t +\eps}{2}} f \|_{L_\infty(\R^d)}   \, dt 
\\
& \lesssim \int_0^ \infty e^{-t} (t+\eps)^{-1/2}\|   \nabla \cP_{\frac{t +\eps}{2}} f \|_{L_\infty(\R^d)}   \, dt 
\\
& \lesssim \eps^{-1/2} \| \cP_ {\frac{1}{2}} f\|_{C^1(\R^d)} 
\\
& \lesssim  \eps^{\beta/2} \|   f \|_{C^\beta(\R^d)}.              \label{eq:whatever}
\end{equs}
 Consequently, 
$$
[\D_i (1-\Delta)^{-1}f ] _{C^{\beta+1} (\R^d)}  \lesssim   \|   f \|_{C^\beta(\R^d)}.
$$
This combined with $\| (1-\Delta)^{-1}f \|_{C^1 (\R^d)} \lesssim   \| f \|_{C^\beta(\R^d)}$ shows that 
\begin{equs}
\| (1-\Delta)^{-1}f \|_{C^{2+\beta} (\R^d)} \lesssim   \| f \|_{C^\beta(\R^d)}.
\end{equs}
The case $-2< \beta< -1$ is treated similarly as above. One can easily see that 
$ \| (1-\Delta)^{-1}f \|_{L_\infty (\R^d)} \lesssim   \| f \|_{C^\beta(\R^d)}$ and what is left to be checked is that  $[ (1-\Delta)^{-1}f ]_{C^{\beta+2} (\R^d)} \lesssim   \| f \|_{C^\beta(\R^d)}$. This is very similar to the previous argument, with the difference that one uses 
$$
\| (1-\Delta)^{-1}f \|_{C^{2+\beta} (\R^d)}  \lesssim  \sup_{ \eps >0} \eps^{-\beta/2}\|\nabla^2  \cP_\eps  (1-\Delta)^{-1}f \|_{L_\infty(\R^d)},
$$
(see, e.g., \cite[Exercise 3.3.6, p.40]{KrylovHolder}) and that  in order to get a bound similar to \eqref{eq:whatever},  one needs the bound $\| \nabla^2 \cP_tg\|_{L_\infty} \leq N t^{-1} \|g\|_{L_\infty}$ for all $t >0$.

 Finally, for the case $ \beta < -2$,  we have for $\eps \in (0,1]$
\begin{equs}
&\| \cP_\eps (1-\Delta)^{-1}f \|_{L_\infty(\R^d)}
\\ & \leq \int_0^{1-\eps} e^{-t} \| \cP_{t+\eps} f \|_{L_\infty(\R^d)}+ \int_{1-\eps}^ \infty  e^{-t} \| \cP_{t+\eps} f \|_{L_\infty(\R^d)}
\\
& \lesssim  \int_0^{1-\eps} (t+\eps)^{\beta/2} \sup_{t \in (0, 1-\eps)} \big( (t+\eps)^{-\beta/2}\|\cP_{t+ \eps} 
f \|_{L_\infty(\R^d)} \big)  \, dt  + \| f \|_{C^\beta(\R^d)}
\\
& \lesssim  \eps^{1+\beta/2}\| f \|_{C^\beta(\R^d)},
\end{equs}
which of course implies that 
\begin{equs}
\|  (1-\Delta)^{-1}f \|_{C^{\beta+2}(\R^d)}  \lesssim \| f \|_{C^\beta(\R^d)}.
\end{equs}
This finishes the proof. 
\end{proof}

\begin{lemma}\label{lem:Lp-bound}
Let $t> s>0$, $\alpha\in(0,1)$, $m\in[1,\infty)$, and $p\in[1,m]$. Then there exist constants $N$ and $N'=N'(\alpha)$ such that the following bounds hold for all $f\in L_m(\R^d)$ and $g\in \dot{W}^{\alpha}_{m}(\R^d)$:
\begin{equs}
\|f(B_t)\|_{L_p(\Omega)}&\leq N t^{-d/(2m)}\|f\|_{L_m(\R^d)},\label{eq:Lp-bound-1}
\\
\| g(B_t)-g(B_s) \|_{L_p(\Omega)} &\leq N' s^{-d/(2m)} |t-s|^{\alpha/2} [g]_{\dot{W}^\alpha_m(\R^d)}.\label{eq:Lp-bound-2}
\end{equs}
\end{lemma}
\begin{proof}
Since $\|\cdot\|_{L_p(\Omega)}\le\|\cdot\|_{L_m(\Omega)}$, it suffices to show the result when $p=m$.
We start with \eqref{eq:Lp-bound-1}. It is evident that
\begin{equs}
\E|f(B_t)|^m=\int_{\R^d}|f(x)|^m p_t(x)\,dx=\big\||f|^m p_t\big\|_{L_1(\R^d)}
\leq\|f\|_{L_m(\R^d)}^m\|p_t\|_{L_{\infty}(\R^d)}.
\end{equs}
By \eqref{eq:HK-Lp-1}, $\|p_t\|_{L_{\infty}(\R^d)}\lesssim  t^{-d/2}$, yielding the claim. Moving on to \eqref{eq:Lp-bound-2}, first one has
\begin{equs}
\| g(B_t)-g(B_s) \|_{L_p(\Omega))}^p =&\int_{\R^d} \int_{\R^d} |g(x)-g(y)|^p p_{t-s}(x-y) p_s(y) \, dx dy\,. \label{trivi}
\end{equs}
If $p=m$, the right-hand side can be bounded by
\begin{equation*}
[ g]_{\dot{W}^\alpha_m(\R^d)}^m\| (\cdot)^{\alpha m +d} p_{t-s}(\cdot)\|_{L_{\infty}(\Rd)} \|p_s\|_{L_{\infty}(\Rd)}
\lesssim [ g]_{\dot{W}^\alpha_m(\R^d)}^m|t-s|^{\alpha m/2}s^{-d/2},
\end{equation*}
using \eqref{eq:HK-Lp-1} and \eqref{eq:HK-Lp-2}. Thus, \eqref{eq:Lp-bound-2} follows.
\end{proof}

\begin{proposition} \label{pro:time-reg}
Let  $p \in [1, \infty)$, $\delta\in(0,1)$ and $\alpha\in(0, 2 \delta]$. Then,  one has for all $f\in \dot{W}^{\alpha}_p(\R^d)$
\begin{equ}\label{eq:HK-bound-1}
\|\cP_tf-\cP_sf\|_{L_p(\R^d)}\leq N|t-s|^{\delta}s^{ \alpha/2-\delta} [f]_{\dot{W}^\alpha_p(\R^d)},
\end{equ}
for all $0 < s < t$, 
where $N$ is a constant depending only on $d,p,\alpha,\delta$. Moreover, for $\alpha = 2\delta$ the estimate also holds with $s=0$ with the convention $\cP_0 f=f$.
\end{proposition}

\begin{proof}
One has
\begin{align*}
\|\cP_tf-\cP_sf\|_{L_p(\R^d)}\leq \int_s^t
\big\|\frac{\partial}{\partial r}\cP_rf\big\|_{L_p(\R^d)}\,dr
&=\int_s^t
\big\|\Delta \cP_r  f\big\|_{L_p(\R^d)}\,dr
\\
&\lesssim [f]_{\dot{W}^\alpha_p(\R^d)}\int_s^t r^{\frac\alpha 2-\delta}r^{-1+\delta}\,dr\\
&\lesssim [f]_{\dot{W}^\alpha_p(\R^d)}s^{\frac\alpha2-\delta}(t-s)^{\delta},
\end{align*}
where we have used Lemma \ref{lem:hessian-seminorm}.  If $\alpha= 2 \delta$, then letting $s\to 0$ and applying Fatou's lemma yields the inequality for $s=0$.
\end{proof}

\begin{proposition}\label{prop:HK-Sigmas}
Let $d\in\N, K>0$ and let $\Sigma,\bar\Sigma$ be symmetric invertible  matrices such that $K^{-1}I\prec\Sigma\bar\Sigma^{-1}\prec KI$. Then for all $x,y\in\R^d$ and $\alpha\in[0,1]$ one has the bounds
\begin{equs}
|p_\Sigma(x)-p_\Sigma(y)|&\leq N |x-y|^\alpha \|\Sigma^{-1}\|^{\alpha/2}\big(p_{\Sigma/2}(x)+p_{\Sigma/2}(y)\big),\label{eq:HK-bound-35}
\\
|p_\Sigma(x)-p_{\bar\Sigma}(x)|&\leq N 
\|I-\Sigma\bar\Sigma^{-1}\|
\big(p_{\Sigma/2}(x)+p_{\bar\Sigma/2}(x)\big),
\label{eq:tired}
\end{equs}
where $N$ is a constant depending only on $d,K$.
\end{proposition}
\begin{proof}
We start with \eqref{eq:HK-bound-35}. The case $\alpha=0$ is trivial, and therefore it suffices to show the bound in the case $\alpha=1$, the remaining cases follow by interpolation. For all $k=1,\ldots,d$, one has
\begin{equ}
\d_{x_k}p_\Sigma(x)=\d_{x_k}\big(-\tfrac{1}{2}x^*\Sigma^{-1}x\big)p_\Sigma(x).
\end{equ}
It is easy to see that $\big|\d_{x_k}|x_\Sigma|^{2}\big|\leq 2|x_\Sigma|\|\Sigma^{-1/2}\|$, and therefore by \eqref{eq:easy-gaussian} one gets
\begin{equ}
|\nabla p_{\Sigma}(x)|\lesssim \|\Sigma^{-1}\|^{1/2}|p_{\Sigma/2}(x)|.
\end{equ}
Now take $x,y\in\R^d$, and assume without loss of generality that $|x_\Sigma| \leq |y_\Sigma|$. Define $\bar x$ to be the minimizer of the distance from $y$ to the set $\{z:|z_\Sigma|=|x_\Sigma|\}$.
By definition, $|\bar x-y|\leq |x-y|$ and every point $z$ on the line segment between $\bar x$ and $y$ satisfies $|z_\Sigma|\geq|x_\Sigma|$.
Moreover, $|\bar x_\Sigma|=|x_\Sigma|$ and therefore $p_\Sigma(x)=p_\Sigma(\bar x)$. Denoting by $e=\frac{y-\bar x}{|y-\bar x|}$ the unit vector in the direction of $y-\bar x$, one has
\begin{equs}
|p_\Sigma(x)-p_\Sigma(y)|=|p_\Sigma(\bar x)-p_\Sigma(y)|
&\leq |y-\bar x|\int_0^1\big|\d_e p_{\Sigma}\big(\bar x+\theta(y-\bar x)\big)\big|\,d\theta\\
&\les |y-\bar x|\int_0^1\|\Sigma^{-1}\|^{1/2}p_{\Sigma/2}\big(\bar x+\theta(y-\bar x)\big)\,d\theta\\
&\lesssim  |y-x|\|\Sigma^{-1}\|^{1/2} p_{\Sigma/2}(x),
\end{equs}
as claimed. Moving on to \eqref{eq:tired},
\begin{equs}
p_\Sigma(x)-p_{\bar\Sigma}(x)&=\big(1-(\det \Sigma\bar\Sigma^{-1})^{1/2}\big)p_\Sigma(x)
\\
&\qquad+\frac{(\det\bar\Sigma^{-1})^{1/2}}{(2\pi)^{d/2}}\Big(\exp(-\tfrac{1}{2}|x_\Sigma|^2)-\exp(-\tfrac{1}{2}|x_{\bar\Sigma}|^2)\Big).
\end{equs}
Thanks to \eqref{eq:simple-matrix-bound} applied with $A=\Sigma\bar\Sigma^{-1}-I$, the first term can be immediately seen to be bounded by the right-hand side of \eqref{eq:tired}. Concerning the second, one can write
\begin{equs}
(\det  \bar\Sigma^{-1})^{1/2}&\big|\exp(-\tfrac{1}{2}|x_\Sigma|^2)-\exp(-\tfrac{1}{2}|x_{\bar\Sigma}|^2)\big|
\\
&\lesssim (\det\bar\Sigma^{-1})^{1/2}
(|x_\Sigma|^2-|x_{\bar\Sigma}|^2)\big|\exp(-\tfrac{1}{2}|x|_{\Sigma}^2)+\exp(-\tfrac{1}{2}|x_{\bar\Sigma}|^2)\big|
\\&\lesssim 
(|x_\Sigma|^2-|x_{\bar\Sigma}|^2)\big(p_{\Sigma}(x)+p_{\bar\Sigma}(x)\big).
\end{equs}
Note that
\begin{align*}
	|x_\Sigma|^2-|x_{\bar \Sigma}|^2=x_\Sigma^*(I- (\Sigma^{1/2})^*\bar \Sigma^{-1}\Sigma^{1/2})x_\Sigma.
\end{align*}
The matrix $A:=I- (\Sigma^{1/2})^*\bar \Sigma^{-1}\Sigma^{1/2}$ is real symmetric and hence has the form $A=Q^*DQ$ where $Q$ is an orthogonal matrix and $D$ is a diagonal matrix. Hence, we have
\begin{align*}
x_\Sigma^*Ax_\Sigma= (Qx_\Sigma)^*D(Q x_\Sigma)\le|Q x_\Sigma|^2\tr(D)=|x_\Sigma|^2\tr(A).
\end{align*}
Furthermore, writing $A=I- \Sigma^{-1/2}\Sigma\bar \Sigma^{-1}\Sigma^{1/2}$, we obtain that
\begin{align*}
	\tr(A)
	=\tr(I- \Sigma^{-1/2}\Sigma\bar \Sigma^{-1}\Sigma^{1/2})
	=\tr(I- \Sigma\bar \Sigma^{-1})
	\les\|I- \Sigma\bar \Sigma^{-1}\|.
\end{align*}
Hence, we have
\begin{align*}
	(\det  \bar\Sigma^{-1})^{1/2}&\big|\exp(-\tfrac{1}{2}|x_\Sigma|^2)-\exp(-\tfrac{1}{2}|x_{\bar\Sigma}|^2)\big|\les \|I- \Sigma\bar \Sigma^{-1}\||x_\Sigma|^2\big(p_{\Sigma}(x)+p_{\bar\Sigma}(x)\big).
\end{align*}
By our assumptions on $\Sigma \bar \Sigma^{-1}$, we also have $|x_\Sigma|\les|x_{\bar \Sigma}|$. Using \eqref{eq:easy-gaussian} once again, we see that
\begin{align*}
	|x_\Sigma|^2\big(p_{\Sigma}(x)+p_{\bar\Sigma}(x)\big)\les|x_\Sigma|^2p_{\Sigma}(x)+|x_{\bar \Sigma}|^2 p_{\bar\Sigma}(x)
	\les p_{\Sigma/2}(x)+p_{\bar\Sigma/2}(x).
\end{align*}
This completes the proof.
\end{proof}

\begin{corollary}
Let $\Sigma_0$ be a symmetric and positive definite $d\times d$-matrix and $x\in\R^d$. Then the function $\Sigma\mapsto p_{\Sigma}(x)$ is differentiable at $\Sigma_0$ and there exists a constant $N$ depending only on $d$ such that
\begin{equ}\label{eq:HK-diff-Sigma}
\|\d_{\Sigma} p_{\Sigma_0}(x)\|\leq N\|\Sigma_0^{-1}\|p_{\Sigma_0/2}(x).
\end{equ}
\end{corollary}
\begin{proof}
The differentiability is obvious, therefore we only need to show the bound \eqref{eq:HK-diff-Sigma}. 
Take $\Sigma'$ to be an arbitrary but fixed matrix with unit norm and $h>0$.
Set $\Sigma=\Sigma_0+h\Sigma'$ and $\bar \Sigma=\Sigma_0$. Notice that for small enough $h$, these two matrices satisfy the condition of Proposition \ref{prop:HK-Sigmas} with, say, $K=2$.
 Applying \eqref{eq:tired}, we get
\begin{equ}
|p_{\Sigma_0}(x)-p_{\Sigma_0+h\Sigma'}(x)|\lesssim \|h\Sigma'\Sigma_0^{-1}\|
|p_{\Sigma_0/2}(x)+p_{(\Sigma_0+h\Sigma')/2}(x)|.
\end{equ}
Upon dividing by $h$ and letting $h\to 0$, we get that $\langle  \D_{\Sigma} p_{\Sigma_0}(x), \Sigma' \rangle$, 
the directional derivative of $p_{\cdot}(x)$ at $\Sigma_0$ in the direction of $\Sigma'$,  satisfies
\begin{equs}
|\langle  \D_{\Sigma} p_{\Sigma_0}(x), \Sigma' \rangle | \lesssim \|\Sigma_0^{-1}\|p_{\Sigma_0/2}(x).
\end{equs}
Taking suprema over $\Sigma'$ gives \eqref{eq:HK-diff-Sigma}.
\end{proof}

\subsection{Sewing}
We need a weighted version of the stochastic sewing lemma.
Such a version appeared recently in \cite{ABML}, but our formulation below allows for a larger range of exponents $\delta_1,\delta_2$.
 Set $[S,T]_\leq=\{(s,t):\,S< s<t\leq T\}$ and 
$$
[S,T]_\leq^*=\{(s,t):\,S<s<t\leq T,\,|t-s|\leq S\}.
$$
For functions $\cA$ of one variable we use the notation $\cA_{s,t}=\cA_{t}-\cA_s$ and for functions $A$ of two variables we use the notation $\delta A_{s,u,t}=A_{s,t}-A_{s,u}-A_{u,t}$.
On $(\Omega,\cF,(\cF_t)_{t\geq 0},\bP)$, the conditional expectation with respect to $\cff_s$ is denoted by $\E^s$.

\begin{lemma}\label{lem:wSSL}
Fix $p\geq2$ and let $A$ be a function from $[0,1]_\leq$ to $L_p(\Omega)$ such that $A_{s,t}$ is $\cF_t$-measurable for all $(s,t)\in[0,1]_{\leq}$.
Suppose that there exist $\eps_1,\eps_2>0,\delta_1,\delta_2\geq 0$ and $C_1,C_2<\infty$ satisfying $1/2+\eps_1-\delta_1>0$, $1+\eps_2-\delta_2>0$ and the following:
for all $S \in(0,1]$, $(s,t)\in[S,1]_\leq^*$ and  $u\in[s,t]$, the bounds
	\begin{equs}
	\|A_{s,t}\|_{L_p(\Omega)} & \leq C_1 S^{-\delta_1}|t-s|^{1/2+\eps_1}, \label{eq:SSL 1}
	\\
	\|\E^s\delta A_{s,u,t}\|_{L_p(\Omega)} & \leq C_2 S^{-\delta_2}|t-s|^{1+\eps_2} \label{eq:SSL2}
	\end{equs}
	hold.
	Then there exists a unique $(\cF_t)_{t\in[0,1]}$-adapted function $\cA:[0,1]\to L_p(\Omega)$ such that $\cA_0=0$ and for all $S\in(0,1]$, there exist constants $K_1,K_2$ such that for all $(s,t) \in[S, 1] _{\leq}^*$ one has
\begin{equs}
\|\cA_{s,t}-A_{s,t}\|_{L_p(\Omega)}&\leq K_1|t-s|^{1/2+\eps_1}+K_2 |t-s|^{1+\eps_2},\label{eq:SSL3}
\\
\|\E^s(\cA_{s,t}-A_{s,t})\|_{L_p(\Omega)}&\leq  K_2|t-s|^{1+\eps_2}.\label{eq:SSL4}
\end{equs}
Furthermore, the above bounds hold with $K_1=N C_1 S^{-\delta_1}$, $K_2=N C_2 S^{-\delta_2}$, where $N$ depends only on $p, \eps_1, \eps_2, \delta_1, \delta_2$. 
	Moreover, $\cA$ satisfies the bounds, for all $S\in(0,1]$, $(s,t) \in[S, 1] _{\leq}^*$,
	\begin{equ}\label{eq:SSL-conclusion}
	\|\cA_{s,t}\|_{L_p(\Omega)}\leq N\big(C_1 S^{-\delta_1}|t-s|^{1/2+\eps_1}+C_2 S^{-\delta_2}|t-s|^{1+\eps_2}\big).
	\end{equ}
	\end{lemma}

\begin{proof}
Note that our assumptions imply that the conditions of the usual stochastic sewing lemma from \cite{Khoa} are satisfied on each interval $I_n:=[2^{-n},2^{-n+1}]$. Therefore we get adapted processes $\cA^{(n)}$ on $I_n$ such that $\cA^{(n)}_{2^{-n}}=0$ and that for all $s,t\in I_n$ one has the bounds \eqref{eq:SSL3}-\eqref{eq:SSL4}-\eqref{eq:SSL-conclusion} hold with $\cA^{(n)}$ in place of $\cA$ and $2^{-n}$ in place of $S$, with $K_1,K_2$ as specified in the statement. In particular, one gets
\begin{equ}
\|\cA^{(n)}_{2^{-n},2^{-n+1}}\|_{L_p(\Omega)}\lesssim (C_1+C_2)2^{-n\kappa}
\end{equ}
with $\kappa:=\min(1/2+\eps_1-\delta_1,1+\eps_2-\delta_2)$. By assumption $\kappa>0$, and therefore if
we define
\begin{equ}
\cA_t=\sum_{n=1}^\infty \bone_{t \geq 2^{-n}} \cA_{2^{-n},2^{-n+1}\wedge t}^{(n)},
\end{equ}
then the sum converges in $L_p(\Omega)$ so the definition is indeed meaningful.
Clearly $\cA$ is adapted, $\cA_0=0$, and since $\cA_{s,t}=\cA^{(n)}_{s,t}$ for $(s,t)\in I_n$, the bounds \eqref{eq:SSL3}-\eqref{eq:SSL-conclusion} hold for $(s,t)\in I_n$, $S=2^{-n}$.
Extending it to general pairs $(s,t)$ is then standard, see, e.g., \cite{Carlo}.
\end{proof}

\subsection{PDE estimates} 
\label{sub:pde_estimates}

In this section, we obtain estimates concerning the PDE associated with \eqref{eq:main}.
Although such results seem folklore, due to the lack of exact reference to our knowledge, we provide a short proof.
In the following, for a function $u$ on $\R^d$, $\nabla^2 u$ is understood to be a matrix. For matrices $A,B$, just like for any other vectors, the $\cdot$ notation denotes the scalar product $A\cdot B=\sum_{i,j}A_{i,j}B_{i,j}$.
Consider the PDE
\begin{equs} \label{eq:PDE}
\begin{aligned}
\D_t u & = \frac{1}{2}(\sigma \sigma^*)\cdot  \nabla^2 u +b\cdot\nabla u -\theta u + f, \qquad && \text{in} \  (0,1) \times \R^d,
\\
u&=0 \qquad && \text{on} \  \{0\} \ \times \R^d,
\end{aligned}
\end{equs}
for $\theta >0$. We look for solutions in the mixed order Sobolev spaces $W^{1,2}_p((0,1) \times \R^d)$ defined as the completion of the set of compactly supported smooth functions with respect to the norm
\begin{equ}
\|u\|_{W^{1,2}_p((0,1) \times \R^d)}=\sum_{k\in\N_0,\ell\in\N_0^d,2k+|\ell|_1\leq 2}\|\d^k_t\d_x^\ell u\|_{L_p((0,1)\times\R^d)}.
\end{equ}

\begin{lemma}          \label{lem:PDE-estimates}
Let  $f \in L_p((0,1) \times \R^d)\cap L_\infty((0,1) \times\R^d)  $ for $p >1$, and   $b \in L_\infty(\R^d)$.  
Then, \eqref{eq:PDE}  has a unique solution $u$  in the class $W^{1,2}_p((0,1) \times \R^d)$.  Moreover, there exists a constant $N$ depending only on $p, \lambda$, $\|\sigma\|_{C^2}$, $T$, and $\|b\|_{\bB(\R^d)}$, such   that 
\begin{equs}              \label{eq:u_xx-estimate}
 \|u \|_{W^{1,2}_p((0,1) \times \R^d) } &  \leq  N \| f\|_{L_p((0,1) \times \R^d)}. 
 \end{equs}
In addition, there  exists $\theta_0>0$ such that for all $\theta> \theta_0$
 \begin{equs}                    \label{eq:u_x-estimate}
 \| \nabla u \|_{\bB( [0, 1] \times \R^d)} & \leq  N {\theta}^{-1/2} \| f\|_{\bB( [0, 1] \times \R^d)}.
\end{equs}
\end{lemma}

\begin{proof}
The existence, the uniqueness, and estimate \eqref{eq:u_xx-estimate} follow from \cite[Theorem 10, p.123]{Krylov_PDE}.  For \eqref{eq:u_x-estimate}, let us assume first that $b=0$. We have that $\tilde{u}(t,x):= e^{\theta t} u(t,x)$ satisfies
\begin{equs}             \label{eq:transformed-PDE}
\D_t \tilde u  = \frac{1}{2}(\sigma \sigma^* )\cdot \nabla^2 \tilde u+ e^{\theta t} f.
\end{equs}
Hence, we have 
\begin{equs}
\nabla \tilde{u}(t, x)= \int_0^t \int_{\R^d} \nabla_x p^{\sigma}_{t-s}(x,y) e^{\theta s} f(s, y) \,dy  ds,
\end{equs}
where  $p^{\sigma}(x,y)$ is the Green's function of the operator $\d_t-\frac{1}{2}(\sigma\sigma^*)\cdot\nabla^2$. 
It is well known (see, e.g., \cite[(6.13), p. 24] {friedman}), that there is a constant $N$ depending on $d$ and $\|\sigma\|_{C^2}$ such that 
$\sup_{x \in \R^d}\| \nabla_x p^\sigma_t(x,\cdot) \|_{L_1(\R^d)} \leq N t^{-1/2}$ for all $t\in(0,1]$, which implies 
\begin{equs}
|\nabla u(t, x)| \lesssim   \| f\|_{\bB([0,1] \times \R^d)} \int_0^t |t-s|^{-1/2}  e^{-\theta (t-s) }  ds \lesssim\theta^{-1/2}  \| f\|_{\bB([0,1] \times \R^d)},
\end{equs} 
where the last inequality can be easily seen from a change of variables. This shows \eqref{eq:u_x-estimate} in the case $b=0$. For the general case, notice that by assumption we have that $f \in L_q((0,1)\times \R^d)$ for all $q \in [p,  \infty)$,  which shows that in fact \eqref{eq:u_xx-estimate} holds for all such $q$ in place of $p$.  From this and the embedding $W^{1,2}_q((0,1) \times \R^d) \subset C^{1-(d+2)/(2q), 2-(d+2)/q}((0,1) \times \R^d)$ for $q$ large enough, it follows that  $b\cdot\nabla u+f \in  L_p((0,1) \times \R^d)\cap L_\infty((0,1) \times\R^d)  $. Hence, we can replace $f$ in our previous analysis   by $b\cdot\nabla u+f$, which gives 
\begin{equs}
 \| \nabla u \|_{\bB( [0, 1] \times \R^d)} & \lesssim  {\theta}^{-1/2} \| f\|_{\bB([0,1] \times \R^d)}+  {\theta}^{-1/2} \| \nabla u  \|_{\bB([0,1] \times \R^d)}< \infty.
\end{equs}
If $\theta$ is large enough, \eqref{eq:u_x-estimate} clearly follows.
\end{proof}

\subsection{A Gronwall-type lemma}
We use a somewhat nonstandard form of Gronwall's lemma, stated in the following lemma.
\begin{lemma} \label{lem:Gronwall-type}
Let $(Y)_{t \in [0,1]}$  be an adapted,  non-negative  process, 
such that $\E \sup_{ t \in [0, 1]} Y ^2_t<\infty$.
Let $(A_t)_{t \in [0,1]}$ be an adapted,  continuous increasing process, and let $C, R \geq 0$ be constants. Assume  that  for all stopping times $\tau$, $\tau'$, with $0 \leq \tau\leq \tau'  \leq 1$ we have
\begin{equs}    \label{eq:ass-gron}
\E \sup_{ t \in [ \tau, \tau']}  Y_t \leq C \E Y _{ \tau} + C R + C \E \int_{\tau}^ {\tau'} Y_s \, dA_s
\end{equs}
Then, there exist a constant $\bar C \geq 0$ depending only on $C$  such that for all $m \in \mathbb{N}$ we have 
\begin{equs}   \label{eq:gronwal-conclusion}
\E \sup_{ t \in [0, 1]} Y _t  \leq \bar C^m \E Y _0   + \bar C^m R +  \bar C (\E \sup_{ t \in [0, 1]} Y ^2_t)^{1/2} ( \mathbf{P} ( A_1 \geq \frac{m}{2C }))^{1/2}.
\end{equs}

\end{lemma}

\begin{proof}
Let us set $\tau_0:=0$  and for $m \in \mathbb{N}$ let us define inductively
\begin{equs}
 \qquad \tau_m:= \inf\{ t > \tau_{m-1} : A_t- A_{\tau_{m-1}} > (2C)^{-1}\}\wedge 1.
\end{equs}
By \eqref{eq:ass-gron} it follows that for all $m \in \mathbb{N}$
\begin{equs}    \label{eq:iterable}
\E \sup_{ t \in [ \tau_{m-1}, \tau_m ]}  Y_t \leq 2C \E Y _{ \tau_{m-1}} +2 C R.
\end{equs}
Consequently, we have 
\begin{equation}    \label{eq:with-sum}
\E \sup_{ t \in [ 0, \tau_m ]}  Y_t \leq \sum_{j=1}^m \E \sup_{ t \in [ \tau_{j-1}, \tau_j ]}  Y_t  \leq  2C \sum_{j=1}^m  \E  Y_{\tau_{j-1}} +  m2CR.
\end{equation}
Notice that \eqref{eq:iterable}, upon iteration,  implies 
\begin{equs}
\E  Y _{ \tau_m}\leq (2C)^m \E Y _0 + R \sum_{j=1}^m (2C)^j \leq \tilde{C}^m \E Y _0 + \tilde{C}^m R,
\end{equs}
for some $\tilde{C}$ depending only on $C$. This combined with \eqref{eq:with-sum}, gives 
\begin{equs}
\E \sup_{ t \in [ 0, \tau_m ]}  Y_t  \leq  2C m \tilde{C} ^m \E Y_0 +2C m \tilde{C}^m R +    m2CR \leq \bar{C}^m \E Y_0 + \bar{C}^mR,
\end{equs}
for an appropriate $\bar{C}$ depending only on $C$. Finally, by H\"older's inequality we have 
\begin{equs}
\E \sup_{ t \in [ 0, 1 ]}  Y_t  \leq & \E \sup_{ t \in [ 0, \tau_m ]}  Y_t+ \big( \E \sup_{ t \in [ 0, 1 ]}  Y_t ^2 \big)^{1/2} \big(  \mathbf{P}(\tau_m < 1) \big)^{1/2}
\\
\leq & \bar{C}^m \E Y_0 + \bar{C}^mR + \big( \E \sup_{ t \in [ 0, 1]}  Y_t ^2 \big)^{1/2} \big(  \mathbf{P}(\tau_m < 1) \big)^{1/2},
\end{equs}
which implies \eqref{eq:gronwal-conclusion}, since $\{ \tau_m <1\} \subset \{A_1 \geq m(2C)^{-1}\}$ by the definition of $\tau_m$. This finishes the proof.
\end{proof}

\section{Integral estimates}\label{sec:quad}
In this section, we obtain estimates with sharp rates for the quantity
\[
	\int_0^1 \big(f_r(X^n_r)-f_r(X^n_{\kappa_n(r)})\big)\, dr
\]
where $f$ is a bounded measurable function on $[0,1]\times\Rd$. The main results of the current section are \cref{cor:Girsanov-multiplicative,cor:Girsanov-additive}. Our method is based on stochastic sewing techniques, \cref{lem:wSSL}. The analytic properties which we utilize to verify the hypotheses of \cref{lem:wSSL} are the heat kernel estimates from Section \ref{sec:HK} and estimates on the density of the Euler--Maruyama approximations, see \eqref{eq:density-bound-1} below. We consider the cases of multiplicative noise and additive noise, corresponding to \cref{asn:multiplicative-basic,asn:additive-basic}, separately.

\subsection{Multiplicative noise}
As in \cite{ButDarGEr}, we first consider the driftless Euler--Maruyama scheme
\begin{equ}\label{eq:EM no drift}
d\bar X^n_t=\sigma(\bar X^n_{\kappa_n(t)})\,dB_t,\quad \bar X^n_0=x_0.
\end{equ}
We will sometimes denote the dependence on the initial condition by writing $\bar X^n_t(x_0)$, but most of the time this dependence will not play any role and therefore will be omitted from the notation.

Let us recall some key estimates for the transition probabilities of the $\bar X^n_t$.
Notice that a trivial induction argument shows that $\bar X^n_t$ has a density for all $t>0$. 
A useful bound for the density in $L_p$ spaces is due to Gy\"ongy and Krylov, see \cref{lem:Krylov-Gyongy-estimate} below. 
Notice that in Lemma \ref{lem:Krylov-Gyongy-estimate} one can not put derivatives on the test function $G$.
For this reason, another bound on the transition probabilities of $\bar X$ was derived in \cite{ButDarGEr}:
under \cref{asn:multiplicative-basic}, by \cite[Theorem~5.1]{ButDarGEr}, for any $G\in C^1$, $t=1/n,2/n,\ldots,1$, one has 
\begin{equ}\label{eq:density-bound-1}
|\E(\nabla G)(\bar X^n_t)|\leq N \|G\|_{\bB}t^{-1/2} + N\|G\|_{C^1}e^{-cn}.
\end{equ}
with some constant $N=N(d,\lambda,\|\sigma\|_{C^2})$ and $c=c(d,\|\sigma\|_{C^2})>0$.

Estimate \eqref{eq:density-bound-1} is applied to obtain the following result, extending \cite[Lem.~6.1]{ButDarGEr} from H\"older continuous functions to bounded measurable functions.
\begin{lemma}\label{lem:(ii) mult}
Let $y\in\R^d$, $\eps\in(0,1/2)$,  $p>0$.
Suppose that \cref{asn:additive-basic} holds and that $\bar X^n:=\bar X^n(y)$ is the solution of \eqref{eq:EM no drift}.
Then for all $f\in \bB([0,1]\times\R^d)$, $0\leq s\leq t\leq 1$, $n\in\N$, one has the bound
\begin{equation}\label{DKBound}
\big\|\int_s^t (f_r(\bar X_r^n)-f_r(\bar X_{\kappa_n(r)}^n))\, dr\big\|_{L_p(\Omega)}
\leq N\|f\|_{\bB([0,1]\times\R^d)} n^{-1/2+2 \eps}|t-s|^{1/2+\eps} ,
\end{equation}
with some $N=N(p, d,\eps,\lambda,\|\sigma\|_{C^2})$.
\end{lemma}
\begin{proof}
Parts of the proof are unchanged from \cite[Lem~6.1]{ButDarGEr}, therefore occasionally we shall refer back to arguments therein.
By the linearity of the left-hand side in $f$, we may and will assume $\|f\|_{\bB([0,1]\times\R^d)}=1$. Set $0\leq s\leq t\leq 1$,
$$
A_{s,t}:=\E^s \int_s^t \big(f_r(\bar X_r^n)-f_r(\bar X_{\kappa_n(r)}^n)\big)\, dr.
$$
We will use Lemma \ref{lem:wSSL} with $\delta_1=\delta_2=0$ (which is the same as the original version of stochastic sewing lemma from \cite{Khoa}).
It is straightforward to verify that $\E^s\delta A_{s,u,t}=0$ for any $0\leq s\leq u\leq t\leq 1$. Therefore \eqref{eq:SSL2} is verified with $C_2=0$.

The bulk of the proof is the verification of \eqref{eq:SSL 1} with 
$C_1=N n^{-1/2+2\eps}$, $\eps_1=\eps$.
Define $k$ by $k/n=\kappa_n(s)$. If $t\in[s,(k+4)/n]$,  then $|t-s|\le 4 n^{-1}$ and it is evident that
\begin{equation}\label{eq:easy-from-BDG}
\|A_{s,t}\|_{L_p(\Omega)}
 \leq 2|t-s|
 \lesssim n^{-1/2+\eps}|t-s|^{1/2+\eps},
\end{equation}
as required.
For the case $t>(k+4)/n$, one first writes
\begin{equs}
|A_{s,t}|= | I_1+I_2|:= \Big|\Big(\int_s^{(k+4)/n}
 +\int_{(k+4)/n}^t\Big) \E^s \big( f_r(\bar X^n_r)-f_r(\bar X^n_{\kappa_n(r)})\big)\, dr\Big|.
\end{equs}
The term $|I_1|$ is easily seen to be bounded by the right-hand-side of \eqref{eq:easy-from-BDG} even for all $\omega$, and so clearly $\|I_1\|_{L_p(\Omega)}\les n^{-1/2+\varepsilon}|t-s|^{1/2+\varepsilon}$. Therefore, \eqref{eq:SSL2} is verified once we show
\begin{equ}\label{eq:sushi}
\|I_2\|_{L_p(\Omega)}\lesssim  |t-s|^{1/2+\eps} n^{-1/2+2\eps}.
\end{equ}
For $r\in[(k+4)/n,t]$, we have $s\le (k+1)/n\le \kappa_n(r)$ so that we can write
\begin{equs}
I_2 =  \E^s \int_{(k+4)/n}^t  \E^{(k+1)/n}\big(\E^{\kappa_n(r)} f_r(\bar X^n_r)-f_r(\bar X^n_{\kappa_n(r)})\big) \, dr.
\end{equs}
Denote $C(x)=(\sigma\sigma^*)(x)$. We have
\begin{equation*}
\E^{\kappa_n(r)} f_r(\bar X^n_r)=\E^{\kappa_n(r)} f_r\big(\bar X^n_{\kappa_n(r)}+(B_r-B_{\kappa_n(r)})\sigma(\bar X^n_{\kappa_n(r)})\big)=\cP_{C(\bar X^n_{\kappa_n(r)})(r-\kappa_n(r))} f_r  (\bar X^n_{\kappa_n(r)}),
\end{equation*}
so with $g^n_r(x):=\cP_{C(x)(r-\kappa_n(r))}f_r(x)-f_r(x)$
we have
\begin{equ}\label{eq:I2 mult}
I_2=\E^s\int_{(k+4)/n}^t\E^{(k+1)/n}g^n_r(\bar X^n_{\kappa_n(r)})\,dr.
\end{equ}
By \cite[Eq.~(6.7)]{ButDarGEr}, one has the bound
\begin{equ}\label{eq: g bound}
\|g_r^n\|_{C^\beta}\lesssim \|f\|_{\bB}\, n^{\beta/2}= n^{\beta/2}
\end{equ}
for all $\beta\in[-1,0]$, $r\in[0,1]$ and $n\in\N$.
Define the operator $\tilde \cP$ by $(\tilde\cP h)(x)=\E h(\bar X^n_{1/n}(x))$ and denote $\tilde g=\tilde \cP g$.
By the tower rule and applying the Markov property twice one gets
\begin{equs}
\E^{(k+1)/n}g^n_r(\bar X^n_{\kappa_n(r)})&=\E^{(k+1)/n}\E^{\kappa_n(r)-1/n}g^n_r(\bar X^n_{\kappa_n(r)})
\\
&=\E^{(k+1)/n}\tilde g^n_r(\bar X^n_{\kappa_n(r)-1/n})
\\
&=\big(\E\tilde g^n_r(\bar X^n_{\kappa_n(r)-(k+2)/n}(y))\big)\big|_{y=\bar X^n_{(k+1)/n}}.\label{i guess}
\end{equs}
Introduce the functions $u$ and $\tilde u$ as the solutions of the equations
\begin{equ}
(1-\Delta)u=g,\qquad(1-\Delta)\tilde u=\tilde g.
\end{equ}
It follows from Lemma \ref{lem:schauder} that these solutions exist, are unique, and satisfy the bounds
\begin{equ}\label{i guess 2}
\|u\|_{C^{\beta+2}}\leq N(\beta)\|g\|_{C^\beta},\qquad \|\tilde u\|_{C^{\beta+2}}\leq N(\beta)\|\tilde g\|_{C^\beta}
\end{equ}
for all $\beta\in\R\setminus\Z$.
Denote $\delta=\kappa_n(r)-(k+2)/n$.
One then has, by  \eqref{i guess}, \eqref{eq:density-bound-1}, and \eqref{i guess 2} 
\begin{equs}
|\E^{(k+1)/n}g(\bar X^n_{\kappa_n(r)})| & \leq  \sup_{y\in\R^d} |\E (\tilde u-\Delta \tilde u)  \big(\bar X^n_{\delta}(y)\big)|
\\& \lesssim \|\tilde u \|_{C^1} \delta^{-1/2}+\|\tilde u \|_{C^2} e^{-cn}
\\& \lesssim  \|\tilde g \|_{C^{-1+\eps}} \delta^{-1/2} + \|\tilde g\|_{C^{\eps}}e^{-cn}.\label{eq:curry}
 \end{equs}
First we find a bound on $\|\tilde g\|_{C^{\eps}}$. Clearly one has $\|\tilde g\|_{C^{0}}\leq\|g\|_{C^{0}}$. 
Furthermore, 
\begin{equs}
|\tilde g(x)-\tilde g(y)|&=\big|\int_{\R^d} \big(p_{\frac{C(x)}{n}}(x-z)-p_{\frac{C(y)}{n}}(y-z)\big)g(z)\,dz\big|
\\
&\leq\|g\|_{\bB}\int_{\R^d}
\big| p_{\frac{C(x)}{n}}(x-z)-p_{\frac{C(x)}{n}}(y-z)\big|
+\big|p_{\frac{C(x)}{n}}(y-z)-p_{\frac{C(y)}{n}}(y-z)\big|\,dz.
\end{equs}
The first term in the integral is bounded via \eqref{eq:HK-bound-35}:
\begin{equ}
\big| p_{\frac{C(x)}{n}}(x-z)-p_{\frac{C(x)}{n}}(y-z)\big|\lesssim
|x-y|^{\eps}n^{\eps/2}\|\big(C(x)\big)^{-1}\|^{\eps/2}\big| p_{\frac{C(x)}{2n}}(x-z)+p_{\frac{C(x)}{2n}}(y-z)\big|.
\end{equ}
By \cref{asn:multiplicative-basic}, we have that
$\|\big(C(x)\big)^{-1}\|\leq N$.
Concerning the other term, we wish to apply \eqref{eq:tired}. To this end, using again \cref{asn:multiplicative-basic}, we have
\begin{equ}
\|I-C(x)\big(C(y)\big)^{-1}\|
\le\|(C(y))^{-1}\|  \|C(y)-C(x)\|\lesssim |x-y|.
\end{equ}
By \eqref{eq:tired} this implies
\begin{equ}
\big|p_{\frac{C(x)}{n}}(y-z)-p_{\frac{C(y)}{n}}(y-z)\big|\lesssim |x-y|
\big|p_{\frac{C(x)}{2n}}(y-z)+p_{\frac{C(y)}{2n}}(y-z)\big|.
\end{equ}
At this point we simply use the fact that any heat kernel $p_{\Sigma}$ has unit integral to conclude that for any $x,y$ with $|x-y|\leq 1$
\begin{equ}
|\tilde g(x)-\tilde g(y)|\lesssim \|g\|_{\bB}|x-y|^{\eps}n^{\eps/2},
\end{equ}
or,
\begin{equ}\label{eq:pasta}
\|\tilde g\|_{C^{\eps}}\lesssim n^{\eps/2}\|g\|_{\bB}.
\end{equ}
Next, we bound $\|\tilde g\|_{C^{-1+\eps}}$. Recall that $\tilde g=\tilde \cP\big((1-\Delta) u\big)$, and one can write the trivial bounds
\begin{equ}\label{eq:risotto}
\|\tilde \cP u\|_{C^{-1+\eps}}\leq\|\tilde \cP u\|_{C^0}\leq \|u\|_{C^{0}}\lesssim  \|g\|_{C^{-1+\eps}}.
\end{equ}
Also 
\begin{equ}\label{eq:taco}
\|\d_k\tilde\cP\d_k u\|_{C^{-1+\eps}}\lesssim \|\tilde \cP\d_k u\|_{C^{\eps}}
\lesssim  n^{\eps/2}\|\d_k u\|_{C^{\eps}}\lesssim n^{\eps/2}\|g\|_{C^{-1+\eps}},
\end{equ}
where in the second inequality we used the previous argument.
Putting $U=\d_k u$, it remains to estimate the commutator $\d_k\tilde\cP U-\tilde\cP\d_k U$ in the norm $C^{-1+\eps}$. It turns out that it can even be bounded in $C^{0}$. Indeed,
\begin{equs}
\big(\d_k\tilde\cP U-\tilde\cP\d_k U\big)(x)&=
\int_{\R^d}\Big(\d_{y_k} p_{\frac{C(y)}{n}}(x-z)\Big)\Big|_{y=x}
U(z)\,dz
\\
&=\int_{\R^d}\Big(\d_\Sigma p_{\Sigma}(x-z)\Big)\Big|_{\Sigma=C(x)/n}\frac{\d_{x_k}C(x)}{n}
U(z)\,dz
\\
&\lesssim\int_{\R^d}p_{\frac{C(x)}{2n}}(x-z)U(z)\,dz
\lesssim \|U\|_{\bB}\lesssim \|u\|_{C^{1+\eps}}\lesssim \|g\|_{C^{-1+\eps}},\label{eq:burger}
\end{equs}
where for $\d_\Sigma p_\Sigma$ we used \eqref{eq:HK-diff-Sigma}.
From \eqref{eq:risotto}, \eqref{eq:taco}, \eqref{eq:burger}, we can conclude
\begin{equ}\label{eq:pizza}
\|\tilde g\|_{C^{-1+\eps}}\lesssim n^{\eps/2}\|g\|_{C^{-1+\eps}}.
\end{equ}
Putting \eqref{eq:pasta} and \eqref{eq:pizza} into \eqref{eq:curry}, and then using \eqref{eq: g bound},
\begin{equs}
|\E^{(k+1)/n}g(\bar X^n_{\kappa_n(r)})|&\lesssim  n^{\eps/2}\delta^{-1/2}\|g\|_{C^{-1+\eps}}+ n^{\eps/2}e^{-cn}\|g\|_{\bB}
\\
&\lesssim  n^{-1/2+\eps}\delta^{-1/2}+Nn^{\eps/2} e^{-cn}.
\end{equs}
Recall that we defined $\delta=\kappa_n(r)-(k+2)/n$ and we are in the situation of \eqref{eq:I2 mult}. In particular, $\delta\geq 2/n$ and so we can further write
\begin{equ}
|\E^{(k+1)/n}g(\bar X^n_{\kappa_n(r)})|\lesssim  n^{-1/2+2\eps}\delta^{-1/2+\eps}.
\end{equ}
Substituting this bound back into \eqref{eq:I2 mult} and integrating, we get \eqref{eq:sushi}.

At this point, the conditions of Lemma \ref{lem:wSSL} are satisfied. It only remains to identify the process $\cA$. We claim that it is given by
\begin{equ}
\hat \cA_t=\int_0^t (f_r(\bar X_r^n)-f_r(\bar X_{\kappa_n(r)}^n))\, dr.
\end{equ}
Clearly $\hat\cA$ is adapted and starts from $0$. Moreover, $A_{s,t}=\E^s\hat\cA_{s,t}$ and therefore \eqref{eq:SSL4} is satisfied with $K_2=0$. On the other hand, one has the trivial bound
\begin{equ}
\|\hat\cA_{s,t}-A_{s,t}\|_{L_p(\Omega)}\leq\|\hat\cA_{s,t}\|_{L_p(\Omega)}+\|A_{s,t}\|_{L_p(\Omega)}\leq 4 |s-t|,
\end{equ}
which verifies \eqref{eq:SSL3} with $K_1=4 $. It therefore follows that $\hat\cA=\cA$, and the bound \eqref{eq:SSL-conclusion} is precisely our claimed bound \eqref{DKBound}.
\end{proof}

\begin{corollary}\label{cor:Girsanov-multiplicative}
Let $\eps\in(0,1/2)$, $p> 0$.
Suppose that \cref{asn:multiplicative-basic} holds and that $X^n$ is the solution of \eqref{eq:main-EM}. Then for all $f\in \bB([0,1]\times\R^d)$, $n\in\N$, one has the bound
\begin{equation}\label{eq:sup-quad-bound-mult}
\big\|\int_0^{\cdot} \big(f_r(X^n_r)-f_r(X^n_{\kappa_n(r)})\big)\, dr\big\|_{L_p(\Omega; \bB[0,1])}
\leq N\|f\|_{\bB([0,1]\times\R^d)} n^{-1/2+2\eps},
\end{equation} 
with some $N=N(p,d,\eps,\lambda,\|b\|_{\bB},\|\sigma\|_{C^2})$.
\end{corollary}
\begin{proof}
Owning to Jensen's inequality, it suffices to prove the statement for $p\geq 2$.
For any continuous process $Z$, let us denote
\begin{equ}
h(Z)=\big\|\int_0^{\cdot} \big(f_r(Z_r)-f_r(Z_{\kappa_n(r)})\big)\, dr\big\|_{ \bB([0,1])}.
\end{equ}
From Kolmogorov's continuity criterion and \eqref{DKBound}, one immediately gets
\begin{equ}\label{eq:cheese}
\|h(\bar X^n)\|_{L_{2p}(\Omega)}\leq   N\|f\|_{\bB([0,1]\times\R^d)} n^{-1/2+2\eps}.
\end{equ} 
Let us set
\begin{equs}
\rho = \exp\left(-\int_0^1  (\sigma^{-1}b)(X_{\kappa_n(r)}^n) \, dB_r - \frac{1}{2}\int_0^1  \big|(\sigma^{-1}b)(X_{\kappa_n(r)}^n)\big|^2 \, dr  \right)
\end{equs}
and define the measure $\tilde\bP$ by  $d \tilde\bP = \rho d \bP$. Since $\sigma^{-1}b$ is a bounded measurable function, $\E \rho^\theta$ is finite for every $\theta\in\R$. 
By Girsanov's theorem, $\tilde\bP$ is a probability measure and $X^n$ solves \eqref{eq:EM no drift} with a $\tilde\bP$-Wiener process $\tilde B$ in place of $B$. In other words, the law of $X^n$ under $\tilde\bP$ and the law of $\bar X^n$ under $\bP$ coincide.
Therefore,
\begin{equs}
\E h(X^n)^p=\tilde \E \big(h(X^n)^p\rho^{-1}\big)\leq\big(\tilde \E h(X^n)^{2p}\big)^{1/2}\big(\tilde\E\rho^{-2}\big)^{1/2}=\big( \E h(\bar X^n)^{2p}\big)^{1/2}\big(\E\rho^{-1}\big)^{1/2}.
\end{equs}
Note that $\E\rho^{-1}$ is bounded by a constant depending only on the supremum of $\sigma^{-1}b$, which in turn is bounded by $\lambda^{-1}\|b\|_{\bB}$. Combining this with \eqref{eq:cheese}, we get the desired bound \eqref{eq:sup-quad-bound-mult}.
\end{proof}

\subsection{Additive noise}
\begin{lemma}\label{lem:(ii)}

Let $\alpha \in (0,1)$, $p \geq 2$, $\eps \in (0, 1/2)$, $\alpha' \in ( 1-2 \eps, 1)$, and $m\geq d$ such that $m > p$. 
Let $f\in \bB([0,1],\dot{W}^{\alpha}_m(\R^d))\cap \bB([0,1]\times\R^d)$ and $g \in \bB([0,1],C^{\alpha'}(\R^d))$ . Then for all $S\in(0,1]$, $(s,t ) \in [S,1]^*_{\leq}$ and $n\in\N$ one has the bound 
\begin{equs}\label{eq:sup-quadr-add}
& \big\|\int_s^t g_r(B_r) \big(  f_r(B_r)-f_r(B_{\kappa_n(r)})\big)\, dr\big\|_{L_p(\Omega)}
\\
\leq &  N\| g\|_{\bB([0,1],C^{\alpha'}(\R^d))}\big(\sup_{r \in [0,1]}[f_r]_{\dot{W}^{\alpha}_m(\R^d)}+\|f\|_{\bB([0,1]\times\R^d)}\big) n^{-(1+\alpha)/2+\eps}|t-s|^{1/2+\eps}S^{- d/(2m) },
\end{equs} 
where 
 $N$ is a constant depending only on 
$d,p,\alpha, m$ and $\eps$.  
\end{lemma} 
\begin{proof}
By linearity of the left-hand side in both $g$ and $f$, we may and will assume 
$$
\| g\|_{\bB([0,1],C^{\alpha'}(\R^d))}=\sup_{r \in [0,1]}[f_r]_{\dot{W}^{\alpha}_m(\R^d)}+\|f\|_{\bB([0,1]\times\R^d)}=1.
$$
We define for $(s,t)\in[0,1]_\leq$
$$
A_{s,t}:=\E^s \int_s^t g_r(B_s)\big(f_r(B_r)-f_r(B_{\kappa_n(r)})\big)\, dr.
$$ 
Let us check the conditions of Lemma \ref{lem:wSSL}, with $\delta_1=\delta_2= d/(2m)$. 
We begin by showing that \eqref{eq:SSL 1} holds with $C_1=N n^{-(1+\alpha)/2+\eps}$, $\eps_1=\eps$.
Take $S\in(0,1]$ and $(s,t)\in[S,1]^\ast_\leq$.
Define $k$ by $k/n=\kappa_n(s)$. 
Suppose first that $t \in [(k+4)/n, 1]$. By using the fact that $\|g_r\|_{C^{\alpha'}(\R^d)} \leq 1$ for all $r \in [0,1]$ and that $p < m$, we have that 
\begin{equs}      \label{eq:A-dom-B}
\|A_{s,t}\|_{L_p(\Omega)} \leq \tilde{A}_{s,t}:= \int_s^t \| \E^s  \big( f_r(B_r)-f_r(B_{\kappa_n(r)})\big) \|_{L_m(\Omega)} \, dr.
\end{equs}
Notice that 
\begin{equs}
\tilde{A}_{s,t}= I_1+I_2:= \Big(\int_s^{(k+4)/n}
 +\int_{(k+4)/n}^t\Big) \|\E^s \big( f_r(B_r)-f_r(B_{\kappa_n(r)})\big)\|_{L_m(\Omega)}\, dr .
\end{equs} 
One has, by \eqref{eq:Lp-bound-1}, \eqref{eq:HK-bound-1}, using $n^{-1}< t-s$ and $[f_r]_{\dot{W}^\alpha_m(\R^d)}\leq 1$ for all $r\in[0,1]$,
\begin{equs}
I_2 &\leq \int_{(k+4)/n}^t\|\cP_{r-s}f_r(B_s)-\cP_{\kappa_n(r)-s}f_r(B_s)\|_{L_m(\Omega)}\,dr
\\
&\lesssim \int_{(k+4)/n}^ts^{-d/(2m)}\|\cP_{r-s}f_r-\cP_{\kappa_n(r)-s}f_r\|_{L_m(\R^d)}\,dr
\\
&\lesssim \int_{(k+4)/n}^tn^{-(1+\alpha)/2}(\kappa_n(r)-s)^{-1/2}s^{-d/(2m)}
\,dr
\\
&\lesssim n^{-(1+\alpha)/2}(t-s)^{1/2}s^{-d/(2m)}
\\
&  \lesssim  n^{-(1+\alpha)/2+\eps}|t-s|^{1/2+\eps}S^{-d/(2m)}.
\end{equs}
Next, we deal with the term $I_1$. If $\alpha \leq  d/m$, we use $\|f_r\|_{\bB}\leq 1$, $r\in[0,1]$, in a trivial way to get
\begin{equs}
I_1 \leq 8 n^{-1}
&\lesssim n^{-(1+\alpha)/2+\eps}|t-s|^{1/2+\eps}S^{-\alpha/2}
\\
& \leq  n^{-(1+\alpha)/2+\eps}|t-s|^{1/2+\eps}S^{-d/(2m)},
\end{equs}
where we also used $n^{-1}\leq|t-s|\leq S$. If $ \alpha> d/m$, then we use that $[f_r]_{C^{\alpha- d/m}(\R^d)}\lesssim [f_r]_{\dot{W}^\alpha_m(\R^d)} $ (see Remark \ref{rem:choice-of-rep}) and we see that 
\begin{equs}
I_1 & \lesssim n^{-1-\alpha/2+ d/(2m)}\|f\|_{\bB([0,1],C^{\alpha-d/m})}
\\
&\leq  n^{-1-\alpha/2+ d/(2m)}
\\
&\leq  n^{-(1+\alpha)/2+\eps}|t-s|^{1/2-d/(2m)+\eps}
\\
& \leq  n^{-(1+\alpha)/2+\eps}|t-s|^{1/2+\eps}S^{-d/(2m)},
\end{equs}
where we have used that  $n^{-1}\leq|t-s|\leq S$ and that $d/m< \alpha < 1$. 
Consequently, for $t \in [(k+4)/n, 1]$, we have shown that 
\begin{equs}
\tilde{A}_{s,t}\lesssim   n^{-(1+\alpha)/2+\eps}|t-s|^{1/2+\eps}S^{-d/(2m)}.
\end{equs}
We now move to the  case $t\in(s,(k+4)/n)$.  We have two subcases, either  $k\geq 1$ or $k=0$. 
Suppose first that $k \geq 1$. We have 
\begin{equs}   
\tilde{A}_{s,t} &=   \int_s^{t \wedge \frac{k+1}{n}} \|\E^s\big( f_r(B_r)-f_r(B_{\kappa_n(r)}) \big)\|_{L_m(\Omega)}  \, dr  
\\&\qquad+ \int_{ t \wedge {\frac{k+1}{n}}}^t \|\E^s\big( f_r(B_r)-f_r(B_{\kappa_n(r)}) \big)\|_{L_m(\Omega)}  \, dr. \label{eq:decom-Ast}
\end{equs}
Next, we see that 
\begin{equs}
\int_s^{t \wedge \frac{k+1}{n}} \|\E^s\big( f_r(B_r)-f_r(B_{\kappa_n(r)}) \big) \|_{L_m(\Omega)} \, dr=  \int_s^{t \wedge \frac{k+1}{n}} \|\cP_{r-s}f_r(B_s)-f_r(B_{k/n}) \|_{L_m(\Omega)} \, dr
 \\
 \leq  \int_s^{t \wedge \frac{k+1}{n}} \|\cP_{r-s}f_r(B_s)-f_r(B_s) \|_{L_m(\Omega)}\, dr+ ({t \wedge \frac{k+1}{n}}-s) \|f_r(B_s)-f_r(B_{k/n})\|_{L_m(\Omega)}.
\end{equs}
For the first term at the right hand side above  we have by  \eqref{eq:Lp-bound-1} and Proposition \ref{pro:time-reg} 
\begin{equs}
\int_s^{t \wedge \frac{k+1}{n}}\|\cP_{r-s}f_r(B_s)-f_r(B_{s})\|_{L_m(\Omega)}  \, dr 
& \lesssim  \int_s^{t \wedge \frac{k+1}{n}} s^{-d/(2m)}   \|\cP_{r-s}f_r-f_r \|_{L_m(\R^d)}  \, dr
\\
& \lesssim  \int_s^{t \wedge \frac{k+1}{n}} s^{-d/(2m)} |r-s|^{\alpha/2}\,dr 
\\
&  \lesssim  |t-s|^{1+(\alpha/2)} s^{-d/(2m)} 
\\
& \lesssim  n^{-(1+\alpha)/2+\eps}|t-s|^{1/2+\eps} S^{-d/(2m)}, 
\end{equs}
where we have used  that $|t-s| \leq 4n^{-1}$ in the last inequality.
For the second term, by \eqref{eq:Lp-bound-2}
\begin{equs}
\Big({t \wedge \frac{k+1}{n}}-s\Big) \|f_r(B_s)-f_r(B_{k/n}) \|_{L_m(\Omega)}  & \lesssim  |t-s| n^{-\alpha/2}  (k/n)^{-d/(2m)} 
\\
& \lesssim  n^{-(1+\alpha)/2+\eps}|t-s|^{1/2+\eps} s^{-d/(2m)} 
\\
& \lesssim  n^{-(1+\alpha)/2+\eps}|t-s|^{1/2+\eps} S^{-d/(2m)}, 
\end{equs}
where we have used that $|t-s| \leq 4 n^{-1}$ and $ k/n \leq s \leq 4k/n$. 
Hence, 
\begin{equs}    
\int_s^{t \wedge \frac{k+1}{n}} \| \E^s \big( f_r(B_r)-f_r(B_{\kappa_n(r)}) \big)\|_{L_m(\Omega)}   \, dr\lesssim  n^{-(1+\alpha)/2+\eps}|t-s|^{1/2+\eps} S^{-d/(2m)}. 
 \label{eq:Ast-term-1}
 \end{equs}
 For the second term at the right hand side of \eqref{eq:decom-Ast} we can assume that $t > (k+1)/n$ and then  we have 
 by \eqref{eq:Lp-bound-2}
\begin{equs}
 \int_{ {\frac{k+1}{n}}}^t \| \E^s\big(  f_r(B_r)-f_r(B_{\kappa_n(r)}) \big)\|_{L_m(\Omega)} \, dr&  \leq  \int_{ {\frac{k+1}{n}}}^t \|  f_r(B_r)-f_r(B_{\kappa_n(r)})\|_{L_m(\Omega)} \, dr 
\\
& \lesssim  \int_{ {\frac{k+1}{n}}}^t |r-\kappa_n(r)|^{\alpha/2} (\kappa_n(r))^{-d/(2m)}\,dr 
\\
& \lesssim  n^{-(1+\alpha)/2+\eps}|t-s|^{1/2+\eps} S^{-d/(2m)}, 
\end{equs} 
where we have used that $|t-s| \leq 4 n^{-1}$ and $ S \leq s \leq (k+1)/n \leq \kappa_n(r)$ for $r \geq (k+1)/n$. Consequently, for the case $t \in (s, (k+4)/n)$ and $k \geq 1$ we get 
\begin{equs}
\tilde{A}_{s,t}
\lesssim  n^{-(1+\alpha)/2+\eps}|t-s|^{1/2+\eps} S^{-d/(2m)}.
\end{equs}
 Finally, for the case  $t \in (s, (k+4)/n)$ and $k =0 $ we have the following. Assume first that $\alpha< d/m$.  From  $ \|f_r\|_{\bB(\R^d)}\leq 1$, $r\in[0,1]$, we have the trivial estimate
 \begin{equ}
\tilde{A}_{s,t} \leq 2 |t-s|.
\end{equ}
By definition of $[S,1]^*_{\leq}$ we have  $|t-s|\leq S \leq s  \leq n^{-1}$, and so one also has 
\begin{equ}
|t-s|\leq n^{-1/2+\eps}|t-s|^{1/2+\eps}\leq n^{-(1+\alpha)/2+\eps}|t-s|^{1/2+\eps}S^{-\alpha/2}\leq  n^{-(1+\alpha)/2+\eps}|t-s|^{1/2+\eps}S^{-d/(2m)} .
\end{equ}
Consequently,
 \begin{equ}
\tilde{A}_{s,t}\leq 2 n^{-(1+\alpha)/2+\eps}|t-s|^{1/2+\eps}S^{-d/(2m)}.
\end{equ}
If $\alpha>d/m$, we can use that $[f_r]_{C^{\alpha- d/m}(\R^d)}\lesssim [f_r]_{\dot{W}^\alpha_m(\R^d)} $ as before. We then get
 \begin{equs}
\tilde{A}_{s,t} & \lesssim n^{-\alpha/2+d/(2m)}|t-s| \sup_{r\in[0,1]}[ f_r]_{C^{\alpha-d/m}(\R^d)}\leq n^{-\alpha/2+d/(2m)}|t-s|.
\end{equs}
Since $S \leq n^{-1}$ we have that $n^{d/(2m)} \leq S^{-d/(2m)}$ which combined with $|t-s| \leq n^{-1}$ gives 
 \begin{equs}
\tilde{A}_{s,t} & \lesssim n^{-(1+\alpha)/2+\eps}|t-s|^{1/2+\eps}S^{-d/(2m)}.
\end{equs}
By combining all of the cases above, we can conclude that the bound 
\begin{equs}      \label{eq:est-tildeA}
\tilde{A}_{s,t} & \lesssim n^{-(1+\alpha)/2+\eps}|t-s|^{1/2+\eps}S^{-d/(2m)}
\end{equs}
holds for all $(s,t)\in[S,1]_\leq^\ast$, which in particular implies \eqref{eq:SSL 1} for $A_{s,t}$ with $C_1=Nn^{-(1+\alpha)/2+\eps}$ (see \eqref{eq:A-dom-B}). 

We now move to \eqref{eq:SSL2}. We will show that it holds with $\eps_2= (\alpha'-1+2\eps)/2>0$. 
Let $(s, t) \in [S,1]^*_{\leq}$ and let $u \in (s,t)$. We have 
\begin{equs}
\E^s \delta  A_{s, u, t} = \int_u^t \E^s \Big( \big( g_r(B_s)-g_r(B_u) \big) \E^u \big( f_r(B_r)-f_r(B_{\kappa_n(r)}) \big) \Big) \, dr,
\end{equs}
which implies that
\begin{equs}
\| \E^s \delta  A_{s, u, t}\|_{L_p(\Omega)} & \leq  \int_u^t \| g_r(B_s)-g_r(B_u) \|_{L_{pm/(m-p)}(\Omega)} \| \E^u \big( f_r(B_r)-f_r(B_{\kappa_n(r)}) \big) \|_{L_m(\Omega)} \, dr
\\
& \lesssim |s-u|^{\alpha'/2} \int_u^t  \| \E^u \big( f_r(B_r)-f_r(B_{\kappa_n(r)}) \big) \|_{L_m(\Omega)} \, dr
\\
& \lesssim   |s-u|^{\alpha'/2}  \tilde{A}_{u, t} 
\\
& \lesssim |t-s|^{\alpha'/2} n^{-(1+\alpha)/2+\eps}|t-s|^{1/2+\eps}S^{-d/(2m)}
\\
&\lesssim n^{-(1+\alpha)/2+\eps}|t-s|^{1+\eps_2}S^{-d/(2m)},
\end{equs}
where the first inequality follows from H\"older inequality, for the last inequality we have used \eqref{eq:est-tildeA} (notice that $(u, t) \in [S,1]^*_{\leq}$). 

Summarising, we have shown that $A_{s,t}$ satisfies the conditions of Lemma \ref{lem:wSSL} with $\delta_1= \delta_2=d/(2m)$, $\eps_1=\eps$, $\eps_2=(\alpha'-1+2\eps)/2$, and $C_1=C_2=  N n^{-(1+\alpha)/2+\eps}$.    Consequently, by  Lemma \ref{lem:wSSL}, there exists a unique process $\cA_t$ satisfying \eqref{eq:SSL3} and \eqref{eq:SSL4}. Let us now set 
$$
\bar{\cA}_t:=\int_0^t g_r(B_r) \big( f_r(B_r)-f_r(B_{\kappa_n(r)}) \big) \, dr.
$$
Since $\|g_r\|_{C^{\alpha'}(\R^d)}, \|f_r\|_{\bB(\R^d)} \leq 1$, for $~r \in [0,1]$, we have the trivial estimates 
\begin{equs}
\| \bar{\cA}_{s,t} - A_{s,t} \|_{L_p(\Omega)} & \leq  4|t-s| \leq 4|t-s|^{1/2+\eps},
\\
\| \E^s( \bar{\cA}_{s,t} - A_{s,t} )\|_{L_p(\Omega)} & \leq 2 |t-s|^{1+\alpha'/2} \leq 2 |t-s|^{1+\eps_2},
\end{equs}
which show that $\bar{\cA}$ satisfies \eqref{eq:SSL3}-\eqref{eq:SSL4}, and therefore $\cA=\bar\cA$.
The desired inequality  \eqref{eq:sup-quadr-add} now follows from 
\eqref{eq:SSL-conclusion}.
\end{proof} 
\begin{lemma}     
Let $\alpha \in (0,1)$, $p \geq 2$, $\eps \in (0, 1/2)$, $\alpha' \in ( 1-2 \eps, 1)$, and $m\geq d$ such that $m > p$. Then, 
for  all  $f\in \bB([0,1],\dot{W}^{\alpha}_m(\R^d))\cap \bB([0,1]\times\R^d)$ and $g \in \bB([0,1],C^{\alpha'}(\R^d))$  one has the bound
\begin{equs}\label{DKBound-frac}
\big\|\int_0^{\cdot} g_r(B_r)\big( & f_r(B_r)-f_r(B_{\kappa_n(r)})\big)\, dr\big\|_{L_p(\Omega; \bB[0,1])}
\\&
\leq N \|g\|_{\bB([0,1],C^{\alpha'}(\R^d)} \big(\sup_{r \in [0,1]}[f_r]_{\dot{W}^{\alpha}_m(\R^d)}+\|f\|_{\bB([0,1]\times\R^d)}\big) n^{-(1+\alpha)/2+\eps},
\end{equs} 
where 
 $N$ is a constant depending only on 
$ d,p,\alpha, m$ and $\eps$.  
\end{lemma} 

\begin{proof}
As before, we may and will assume
$$\|g\|_{\bB([0,1],C^{\alpha'}(\R^d))}=\sup_{r \in [0,1]}[f_r]_{\dot{W}^{\alpha}_m(\R^d)}+\|f\|_{\bB([0,1]\times\R^d)}=1.$$
We will first show that for $q \in (p, m)$ and  all $t \in [0,1]$ we have 
\begin{equs}\label{eq:0-to-t}
\big\|\int_0^t g_r(B_r) \big(f_r(B_r)-f_r(B_{\kappa_n(r)})\big)\, dr\big\|_{L_q(\Omega)}
\lesssim 
n^{-(1+\alpha)/2+\eps}.
\end{equs}  
For any $\ell\in\N$, by Lemma \ref{lem:(ii)} we get
\begin{equs}
 \big\| &\int_{t2^{-\ell}}^t g_r(B_r)   \big(f_r(B_r)-f_r(B_{\kappa_n(r)})\big)\, dr\big\|_{L_q(\Omega)}
\\
 & \leq   \sum_{k=0}^{\ell-1} \big\| \int_{t2^{k-\ell}}^{t2^{k-\ell+1}}  g_r(B_r)  \big(f_r(B_r)-f_r(B_{\kappa_n(r)})\big)\, dr\big\|_{L_q(\Omega)}
\\
 &\lesssim 
 n^{-(1+\alpha)/2+\eps}  \sum_{k=0}^{\ell-1} (t2^{k-\ell})^{1/2+\eps} ( t 2^{k-\ell})^{-d/(2m)}.
\end{equs}
Then, notice that since $d \leq m$, we have 
\begin{equs}
\sum_{k=0}^{\ell-1} (t2^{k-\ell})^{1/2+\eps} ( t 2^{k-\ell})^{-d/(2m)} \leq   2^{-\ell \eps} \sum_{k=0}^{\ell-1}  2^{k\eps} = 2^{-\ell\eps} \frac{2^{\ell\eps}-1}{2^\eps -1} \leq \frac{1}{2^\eps -1} .
\end{equs}
Hence, \eqref{eq:0-to-t} follows from Fatou's lemma by letting $\ell \to \infty$.  Since \eqref{eq:0-to-t} holds for all $t \in [0,1]$,  we also get that for all $t \in [0,1]$ 
\begin{equs}\label{eq:t-to-T}
\big\|\int_t^1 g_r(B_r) \big(f_r(B_r)-f_r(B_{\kappa_n(r)})\big)\, dr\big\|_{L_q(\Omega)}
\lesssim 
n^{-(1+\alpha)/2+\eps}.
\end{equs}
Let $\tau$ be a stopping time bounded by $1$, taking only finitely many values $t_1, t_2, ..., t_k$. We have 
\begin{multline}    
\E \ \Big| \int_\tau ^1  g_r(B_r)  \big(f_r(B_r)-f_r(B_{\kappa_n(r)})\big)\, dr \Big|^q 
\\=  \sum_{i=1}^k \E \Big( \bone_{\tau= t_i} \ \Big| \int_{t_i}^1 g_r(B_r)  \big(f_r(B_r)-f_r(B_{\kappa_n(r)})\big)\, dr \Big|^q \Big).
 \label{eq:values-of-tau}
\end{multline}
Define $\kappa^+_n(t_i) := \kappa_n(t_i) +1/n$.
For each of the summands on the right-hand side, we have 
\begin{equs}       
 \E \Big(& \bone_{\tau= t_i} \ \Big| \int_{t_i} ^1 g_r(B_r) \big(f_r(B_r)-f_r(B_{\kappa_n(r)})\big)\, dr \Big|^q \Big)
\\               
 &\lesssim   \E \Big( \bone_{\tau=t_i } \ \Big| \int_{ \kappa^+_n(t_i)} ^1  g_r(B_r)\big(f_r(B_r)-f_r(B_{\kappa_n(r)})\big)\, dr \Big|^q \Big) + n^{-q}  \mathbf{P}(  \tau=t_i ),
  \label{eq:term-ti}
\end{equs}
using $\|g_r\|_{\bB}, \|f_r\|_{\bB}\leq 1$, $r\in[0,1]$,  in a trivial way.
If $1-t_i \leq 3n^{-1}$, then we have  similarly the trivial estimate
\begin{equs}        \label{eq:case-T-ti-small}
\E \Big( \bone_{\tau=t_i } \ \Big| \int_{ \kappa^+_n(t_i)} ^1 g_r(B_r) \big(f_r(B_r)-f_r(B_{\kappa_n(r)})\big)\, dr \Big|^q \Big)  \lesssim  n^{-q} 
\mathbf{P}(  \tau=t_i ).
\end{equs}
If $1-t_i > 3n^{-1}$ we can write 
\begin{equs}      
\E \Big( &\bone_{\tau=t_i } \ \Big| \int_{ \kappa^+_n(t_i)} ^1 g_r(B_r) \big(f_r(B_r)-f_r(B_{\kappa_n(r)})\big)\, dr \Big|^q \Big)
=  \E \big( \bone_{\tau=t_i } G(B_{\kappa^+_n(t_i)}) \big) ,          \label{eq:case-T-ti-large}
\end{equs}
where
\begin{equs}
G(x)&: =  \E  \Big| \int_{ \kappa^+_n(t_i)} ^1  g_r(B_r-B_{\kappa^+_n(t_i)}+x) \big(f_r(B_r-B_{\kappa^+_n(t_i)}+x)-f_r(B_{\kappa_n(r)}-B_{\kappa^+_n(t_i)}+x) \big) \, dr \Big|^q
\\
&=  \E  \Big| \int_0^ {1- \kappa^+_n(t_i)}  g_{r+\kappa_n^+(t_i)}(B_r+x)\big(f_{r+\kappa_n^+(t_i)}(B_r+x)-f_{r+\kappa_n^+(t_i)}(B_{\kappa_n(r)}+x) \big) \, dr \Big|^q.
\end{equs} 
Hence, by \eqref{eq:0-to-t}, we conclude that in the case $1-t_i> 3n^{-1}$  we have that 
\begin{equs}      
 \E \Big( \bone_{\tau=t_i } \ \Big| \int_{ \kappa^+_n(t_i)} ^1 g_r(B_r)  \big(f_r(B_r)-f_r(B_{\kappa_n(r)})\big)\, dr \Big|^q \Big) & = \E \big( \bone_{\tau=t_i } G(B_{\kappa^+_n(t_i)}) \big) 
\\
& \lesssim 
\big(n^{-(1+\alpha)/2+\eps} \big)^q  \mathbf{P}( \tau=t_i).
\end{equs}
Putting the above inequality together with \eqref{eq:case-T-ti-small}, \eqref{eq:term-ti}, and \eqref{eq:values-of-tau}, gives 
\begin{equs}
\big\|\int_\tau^1 g_r(B_r)  \big(f_r(B_r)-f_r(B_{\kappa_n(r)})\big)\, dr\big\|_{L_q(\Omega)}
\lesssim 
n^{-(1+\alpha)/2+\eps}.\label{eq:ref here}
\end{equs}  
Recall that $\tau \leq 1$ was a  simple stopping time. 
It is well known that an arbitrary stopping time can be approximated by simple ones (for example, one can take $\tau_\ell=\kappa_\ell^+(\tau)$ and let $\ell\to\infty$).
Therefore, a standard approximation argument shows that  \eqref{eq:ref here} holds for all stopping times which are bounded by $1$.
Moreover, the above combined with \eqref{eq:0-to-t} implies that for all such stopping times $\tau$, we have 
\begin{equs}
\big\|\int_0^\tau g_r(B_r)   \big(f_r(B_r)-f_r(B_{\kappa_n(r)})\big)\, dr\big\|_{L_q(\Omega)}
\lesssim 
n^{-(1+\alpha)/2+\eps}.
\end{equs} 
The claimed bound \eqref{DKBound-frac} then follows by Lenglart's inequality (see, e.g., \cite[Proposition~IV.4.7]{RY}). 
\end{proof}

\begin{corollary}\label{cor:Girsanov-additive}
Let $\alpha \in (0,1)$, $p \geq 2$, $\eps \in (0, 1/2)$, $\alpha' \in ( 1-2 \eps, 1)$, and $m\geq d$ such that $m > p$.   
Let \cref{asn:additive-basic} hold and let $X^n$ be the solution of \eqref{eq:main-EM}.
Then, 
for  all $f\in \bB([0,1],\dot{W}^{\alpha}_m(\R^d))\cap \bB([0,1]\times\R^d)$, $g \in \bB([0,1],C^{\alpha'}(\R^d))$,  and $n\in\N$ one has the bound
\begin{equs}\label{DKBound-frac-Girsanov}
\big\|\int_0^{\cdot} g_r(X^n_r)   \big( & f_r(X^n_r)-f_r(X^n_{\kappa_n(r)})\big)\, dr\big\|_{L_p(\Omega; \bB[0,1])}
\\
&
\leq N\| g\|_{ \bB([0,1],C^{\alpha'}(\R^d))}\big(\sup_{r \in [0,1]}[f_r]_{\dot{W}^{\alpha}_m(\R^d)}+\|f\|_{\bB([0,1]\times\R^d)}\big) n^{-(1+\alpha)/2+\eps},
\end{equs} 
 where $N$ is a constant depending only on 
$ d,p,\alpha, m$ and $\eps$.  
\end{corollary}  
The proof is a simple application of Girsanov's theorem and works just like the proof of Corollary \ref{cor:Girsanov-multiplicative}, so we omit repeating the details.

\section{An intermediate stability estimate}\label{sec:pde}
In this section we consider two `approximating solutions' to the main SDE \eqref{eq:main}. More precisely, we assume that we are given adapted continuous processes $X,\bar X,Y,\bar Y$, all of them with initial condition $x_0$, such that
\begin{equs}
dX_t&=b(\bar X_t)\,dt+\sigma(\bar X_t)\,dB_t,
\\
dY_t&=b(\bar Y_t)\,dt+\sigma(\bar Y_t)\,dB_t,
\end{equs}
and such that the laws of $X_t$ and $Y_t$ are absolutely continuous with respect to the Lebesgue measure for all $t>0$. 
We furthermore denote $\hat X=X-\bar X$, $\hat Y=Y-\bar Y$. To relate with Euler--Maruyama scheme \eqref{eq:main-EM}, one may think of $\bar X,\bar Y$ respectively as $X$ and $X^n_{\kappa_n}$.
This kind of reformulation of the error analysis
is inspired by \cite{KrylovSimple,MR1119837,MR1617049}, and in some vague sense, replaces the `regularisation lemma' step from \cite{ButDarGEr}.
Since the coefficients $b$ and $\sigma$ are bounded, we have that there exists $N$ depending only on $d$, $\| b\|_{\bB}, \|\sigma\|_{\bB}$,  such that  
\begin{equs} \label{eq:exponential-bounds-new}
\E  \exp \big( \|X_\cdot\|_{\bB([0,1])} \big) \leq N \exp( x_0) , \qquad  \E \exp \big( \|Y_\cdot\|_{\bB([0,1])} \big) \leq N \exp( x_0).
\end{equs}
Estimating the difference of the drifts is done via a PDE method, similarly to, e.g., \cite{PT, DG, NeuSz, Bao2020}.
First, for $K\in(0,\infty)$, we introduce the truncation $b_K=b\bone_{|x|\leq K}$.
The reason for this truncation is to enforce the right-hand side of the PDE below to be in $L_p(\R^d)$ with $p<\infty$, since Schauder estimates fail in the endpoint $p=\infty$ case.
For $ \ell \in \{1,\ldots,d\}$, and $\theta>0$ to be chosen later, let us consider the equation 
\begin{equs} 
\begin{aligned} \label{eq:PDE-trun-new}
\D_t u^{\ell}+\frac{1}{2}(\sigma \sigma^*)\cdot \nabla^2 u^{\ell} + b\cdot\nabla u^{\ell} -\theta u^{\ell} &  = b^{\ell}_K, \qquad && \text{in} \  (0,1) \times \R^d
\\
u^\ell&=0 \qquad && \text{on} \ \{1\} \ \times \R^d.
\end{aligned}
\end{equs}
By $u$ we denote the $\R^d$-valued function whose coordinates are $u^1,\ldots, u^d$.
Note that by a change of time variable $t\leftrightarrow 1-t$, the estimates in \cref{lem:PDE-estimates} also apply for the backward equation \eqref{eq:PDE-trun-new}.

For a function $f$ let us denote by $\mathcal{M}f$ its Hardy--Littlewood maximal function (see, e.g., \cite{Aalto} for a brief introduction), that is,
\begin{align}\label{def.maximalf}
	\mathcal{M} f(x) := \sup_{r>0} \frac{1}{|B_r(x)|} \int_{B_r(x)} f(y) \, dy , \qquad x \in \R^d.
\end{align}

Given the objects above and $p\geq 2$, introduce the increasing process
\begin{equs}
A_t = t+\int_0^t \big| \big(\mathcal{M} | \nabla ( \nabla u \sigma) | \big)(s,X_s) +   \big(\mathcal{M} | \nabla ( \nabla u \sigma) |\big) (s,Y_s) \big|^p  \, ds.\label{eq:def-A}
\end{equs}
\begin{lemma}\label{lem:kindofregularisation}
Assume the above setting and fix $p\geq 2$, $m\in\N$.
Then there exist constants $N$, $\theta$ depending only on $d,p,\|b\|_{\bB},$ and $\|\sigma\|_{\bB}$ (but not on $K$ and $m$)   such that
\begin{equation}\label{eq:reg-new}
\E\sup_{t\in[0,1]}|X_t-Y_t|^p
\leq N\big(\bP(A_1\geq \frac{m}{2N})\big)^{1/2}+N^m\cR,
\end{equation}
where
\begin{multline}\label{def.cR}
	\cR=\Big(e^{- K}+\bone_{\nabla\sigma\neq 0} \sum_{U=X,Y}\sup_{t\in[0,1]}\big(\E(|\hat U_t|^{2p})\big)^{1/2}\big(1+\E\int_0^1|\nabla^2 u(s,U_s)|^{2p}\,ds\big)^{1/2}
	\\
	+\sum_{U=X,Y}\E\sup_{t\in[0,1]}\Big|\int_0^t \big(b(U_s)-b(\bar U_s)\big)\,ds\Big|^p+\Big|\int_0^t \big(b(U_s)-b(\bar U_s)\big)\nabla u(s,U_s)\,ds\Big|^p\Big),
\end{multline}
and $u$ is a solution to \eqref{eq:PDE-trun-new}.
\end{lemma}

\begin{proof}
Let $\tau_1, \tau_2$ be stopping times with $0 \leq \tau_1 \leq \tau _2 \leq T $.
Denote $Z_t= |X_t-Y_t|^p$. 
Notice that by \eqref{eq:exponential-bounds-new}, and Markov's inequality one has
\begin{equ}
\E\sup_{t\in[\tau_1,\tau_2]}\big| \int_{\tau_1}^t b(X_s)-b_K( X_{s}) \, ds \big|^p \lesssim  \bP(\sup_{t\in[0,1]}|X_t|\geq K)\lesssim  e^{- K},
\end{equ}
and similarly for $Y$.
Therefore, repeated application of the triangle inequality yields
\begin{equs}
\E  \sup_{ t \in [\tau_1, \tau_2]} Z_ t 
&\lesssim  \E Z_{\tau_1} 
+ \E \sup_{t \in [\tau_1, \tau_2]} \big| \int_{\tau_1}^t b(\bar X_s)-b(\bar Y_{s}) \, ds \big|^p 
\\
&\quad+  \E  \sup_{t \in [\tau_1, \tau_2]}   \big| \int_{\tau_1}^{t} \sigma (\bar X_s)-\sigma (\bar Y_s) \, dB_s  \big|^p
\\
&\lesssim \E Z_{\tau_1} 
+ \E \sup_{t \in [\tau_1, \tau_2]} \big| \int_{\tau_1}^t b(X_s)-b( Y_{s}) \, ds \big|^p 
\\
&\quad+  \E  \sup_{t \in [\tau_1, \tau_2]}   \big| \int_{\tau_1}^{t} \sigma (\bar X_s)-\sigma (\bar Y_s) \, dB_s  \big|^p+\cR
\\
&\lesssim  \E Z_{\tau_1} 
+ \E \sup_{t \in [\tau_1, \tau_2]} \big| \int_{\tau_1}^t b_K( X_s)-b_K( Y_{s}) \, ds \big|^p 
\\
&\quad+  \E  \sup_{t \in [\tau_1, \tau_2]}   \big| \int_{\tau_1}^{t} \sigma (\bar X_s)-\sigma (\bar Y_s) \, dB_s  \big|^p+\cR.
   \label{eq:triangle-new}
\end{equs}
By the Burkholder--Davis--Gundy inequality and the Lipschitz continuity of $\sigma$ we have
\begin{equs}
\E  \sup_{t \in [\tau_1, \tau_2]}   \big| \int_{\tau_1}^{t} \sigma (\bar X_s)-\sigma (\bar Y_s) \, dB_s  \big|^p &\lesssim \E\Big(\int_{\tau_1}^{\tau_2}\big(\sigma (\bar X_s)-\sigma (\bar Y_s)\big)^2\,ds\Big)^{p/2}
\\
&\lesssim \E\int_{\tau_1}^{\tau_2}Z_s\,ds+\bone_{\nabla\sigma\neq 0}\sup_{t\in[0,1]}\E\big(|\hat X_t|^p+|\hat Y_t|^p\big).
\end{equs}
Therefore we arrive at
\begin{equ}\label{eq:so-far-quick-new}
\E  \sup_{ t \in [\tau_1, \tau_2]} Z_ t \lesssim \E Z_{\tau_1}+\E\int_{\tau_1}^{\tau_2}Z_s\,ds+ \E \sup_{t \in [\tau_1, \tau_2]} \big| \int_{\tau_1}^t b_K( X_s)-b_K( Y_{s}) \, ds \big|^p +\cR.
\end{equ}
The integral involving $b_K$ is treated via a PDE method.
Although $u^\ell$ is not spatially twice continuously differentiable, one has $u^\ell \in W^{1,2}_q([0,1] \times \R^d)$ for all $q<\infty$ by Lemma \ref{lem:PDE-estimates}. Therefore, It\^o's formula can be applied  (see, e.g., \cite[Theorem~1, p.122]{KrylovControl}.
Hence for any $ \ell  \in \{1,...,d\}$ and $U\in\{X,Y\}$, on $\{ t \geq \tau_1\}$ we have from \eqref{eq:PDE-trun-new} and It\^o formula that 
\begin{align*}
	u^\ell(t,U_t)-u^\ell(\tau_1,U_{\tau_1})&= \int_{\tau_1}^t (\theta u^\ell(s,U_s)+b^\ell_K(U_s))ds+\int_{\tau_1}^t\nabla u^\ell(s,U_s)\sigma(U_s)dB_s
	\\&\quad+\cee^{U,\ell}_1(t) +\cee^{U,\ell}_2(t) +\cee^{U,\ell}_3(t),
\end{align*}
where
\begin{equs}
\mathcal{E}_1^{U,\ell}(t)& := \int_{\tau_1}^{t } \nabla  u^{\ell}(s,U_s)\cdot( b(\bar U_s) -b(U_s) ) \, ds ,
\\
\mathcal{E}_2^{U,\ell}(t) & := \int_{\tau_1}^{ t }  \nabla^2 u^{\ell}(s,U_s)\cdot\big( ( \sigma \sigma^*) (\bar U_s)-( \sigma \sigma^*)(U_s) ) \, ds ,
\\
\mathcal{E}_3^{U,\ell}(t) & := \int_{\tau_1}^{t }  \nabla  u^{\ell}(s,U_s)( \sigma(\bar U_s)-\sigma (U_s) ) \, dB_s.
\end{equs}
It follows that 
\begin{equs}
\int_{\tau_1}^{t } b^{\nell}_K (X_s)- b^{\nell}_K (Y_s)  \, ds & = u^{\nell}(t, X_{t } ) - u^{\nell} (t, Y_{t} )
-u^{\nell}( \tau_1, X_{\tau_1} ) +u^{\nell}( \tau_1, Y_{\tau_1} ) 
\\
&\quad- \int_{\tau_1}^{ t }  \theta \big( u^{\nell}(s, X_s)-u^{\nell}(s, Y_s)\big) \, ds 
\\
 & \quad-\int_{\tau_1}^{ t } \big( \nabla  u^{\nell}(s,X_s) \sigma ( X_s) - \nabla  u^{\nell}(s,Y_s) \sigma ( Y_s)\big) \, dB_s
 \\
&\quad  -\mathcal{E}^X_1(t)-\mathcal{E}^X_2(t)-\mathcal{E}^X_3(t)+
\mathcal{E}^Y_1(t)+\mathcal{E}^Y_2(t)+\mathcal{E}^Y_3(t). \label{eq:after-Ito-new}
\end{equs} 
For the first couple of terms in \eqref{eq:after-Ito-new} we apply \eqref{eq:u_x-estimate}, keeping in mind that $b_K$ (playing the role of $f$ therein) has its $L_\infty$ norm bounded by $\|b\|_{\bB}$, independently of $K$. Therefore, we have $\|\nabla u\|_{\bB([0,1]\times\R^d)}\leq N\theta^{-1/2}$, and so
\begin{equs}     
 \sup_{t \in [\tau_1, \tau_2]} & |  u^{\nell}( t, X_{t } ) - u^{\nell} (t, Y_{t} )
 -u^{\nell}( \tau_1, X_{\tau_1} ) +u^{\nell}( \tau_1, Y_{\tau_1} ) |+ \theta  \int_{\tau_1}^{ t }  | u^{\nell}(s, X_s)-u^{\nell}(s, Y_s) | \, ds 
 \\
& \lesssim   \theta^{-1/2}\sup_{t \in [\tau_1, \tau_2]} | X_t-Y_t|+   \theta^{1/2}  \int_{\tau_1}^{ t }  | X_s- Y_s | \, ds.  \label{eq:estimated-by-grad-new}
\end{equs}
Note that we have $\E\sup_{t\in[\tau_1,\tau_2]} Z_t<\infty$. Therefore, if we combine \eqref{eq:so-far-quick-new}, \eqref{eq:after-Ito-new}, \eqref{eq:estimated-by-grad-new}, and choose $\theta$ to be large enough, we get
\begin{equs}[eq:100]
\E\sup_{t\in[\tau_1,\tau_2]} Z_t&\lesssim \E Z_{\tau_1}+\E\int_{\tau_1}^{\tau_2} Z_s\,ds+\cR
\\
&\quad+\E\sup_{t\in[\tau_1,\tau_2]}\Big|\int_{\tau_1}^{ t }  \nabla  u^{\nell}(s,X_s) \sigma ( X_s) -  \nabla  u^{\nell}(s,Y_s) \sigma ( Y_s) \, dB_s\Big|^p
\\
&\quad+\sum_{i=1,2,3;\, U=X,Y}\E\sup_{t\in[\tau_1,\tau_2]}|\cE^U_i(t)|^p.
\end{equs}
In the sequel, we suppress the time argument from $u$ whenever there is no danger of confusion.
The next term to deal with is the stochastic integral, which requires some care due to the lack of Lipschitz continuity of $\nabla u^{\nell}$.
We argue as in \cite{Bao2020},  using the Hardy–Littlewood maximal function (see \eqref{def.maximalf}).
One then has the following well-known inequality: there exists a constant $N$ depending only on $d$ such that for all $f \in  W^1_{1,loc}(\R^d)$, for almost all $x,y\in\R^d$,
\begin{equs}\label{eq:maximal-new}
|f(x)-f(y)| \leq N |x-y|\big( \mathcal{M}| \nabla f| (x)+\mathcal{M}| \nabla f| (y) \big).
\end{equs}
Recall also the Hardy--Littlewood maximal inequality (\cite{Aalto}),
\begin{equs}\label{eq:H-L-new}
\| \mathcal{M} f\|_{L_p(\R^d)} \leq N \|f\|_{L_p(\R^d)},
\end{equs}
for all $f \in L_p(\R^d)$, $p\in(1,\infty)$, where $N$ depends only on $d$ and $p$.
Since by assumption the laws of $X_t$ and $Y_t$ are absolutely continuous with respect to the Lebesgue measure for all $t>0$, \eqref{eq:maximal-new} holds with $X_t$ and $Y_t$ in place of $x$ and $y$, $d\bP\otimes dt$-almost surely.
Combining this with the Burkholder--Davis--Gundy inequality, we have
\begin{equs}
\E \sup_{t \in [\tau_1, \tau_2]}  & \big| \int_{\tau_1}^{ t }  \nabla  u^{\nell}(X_s) \sigma (X_s) -  \nabla  u^{\nell}(Y_s) \sigma (Y_s) \, dB_s \big| ^p 
\\
&\lesssim  \E \big| \int_{\tau_1}^{\tau_2}  |\nabla  u^{\nell}(X_s) \sigma (X_s) -  \nabla  u^{\nell}(Y_s) \sigma (Y_s) |^2 \, ds \big|^{p/2}
\\
& \lesssim  \E  \int_{\tau_1}^{\tau_2}   |X_s-Y_s| ^p \big| (\mathcal{M} | \nabla ( \nabla u^{\nell} \sigma) | (X_s) +   (\mathcal{M} | \nabla ( \nabla u^{\nell} \sigma) | (Y_s) \big|^p  \, ds 
\\
& \leq  \E  \int_{\tau_1} ^{ \tau_2 }  Z_s   \, d A_s,\label{eq:est-stoch-integ-new}
\end{equs}
where in the last equality, we used the definition \eqref{eq:def-A} of the process $A$.
We now move on to bounding the terms $\cE_i^U$. The bound
\begin{equ}\label{eq:E1-new} 
\E  \sup_{t \in [\tau_1, \tau_2]} | \mathcal{E}_1^U(t)|^p\leq\cR
\end{equ}
is immediate.
 Next, by the Lipschitz continuity of $\sigma$ and H\"older's inequality, we have
 \begin{equs}
 \E\sup_{t\in[\tau_1,\tau_2]}|\cE_2^U(t)|^p &\lesssim \bone_{\nabla \sigma\neq 0} \int_{0}^1\E|\nabla^2 u^\ell(U_s)|^p|\hat U_s|^{p}\,ds
\leq \cR.\label{eq:E2-new}
 \end{equs}
 Finally, by the Burkholder--Davis--Gundy and Jensen's inequalities and the Lipschitz continuity of $\sigma$, we have
 \begin{equs}\label{eq:E3-new}
  \E\sup_{t\in[\tau_1,\tau_2]}|\cE_3^U(t)|^p &\lesssim \bone_{\nabla \sigma\neq 0} \int_{0}^1\|\nabla u^{\nell}\|_{\bB([0,1]\times\R^d)}\E|\hat U_s|^{p}\,ds\leq\cR.
 \end{equs}
 since $\|\nabla u^{\nell}\|_{\bB([0,1]\times\R^d)}\lesssim 1$.
We can now combine \eqref{eq:100}, \eqref{eq:est-stoch-integ-new}, \eqref{eq:E1-new}, \eqref{eq:E2-new}, and \eqref{eq:E3-new} altogether to get
\begin{equ}
\E\sup_{t\in[\tau_1,\tau_2]} Z_t\leq N\E Z_{\tau_1}+N\E\int_{\tau_1}^{\tau_2} Z_s\,d A_s+N\cR.
\end{equ}
This brings us to the setting of Lemma \ref{lem:Gronwall-type}. From \eqref{eq:gronwal-conclusion}, we therefore obtain
\begin{equ}
\E\sup_{t\in[0,1]}Z_t\leq N^m\cR+ N \big(\E\sup_{t\in[0,1]}Z_t^2\big)^{1/2}\big(\bP( A_1\geq \frac{m}{2N})\big)^{1/2}.
\end{equ}
Since by \eqref{eq:exponential-bounds-new}, $\E\sup_{t\in[0,1]}Z_t^2\leq N$, this is precisely the claimed bound.
\end{proof}

\section{Proofs of the main results}\label{sec:proof}
First we recall the following estimate on the density of the Euler--Maruyama scheme due to Gy\"ongy and Krylov \cite[Theorem~4.2]{GyK}. We remark that while therein this bound is proved for the $b=0$ case, the general case follows immediately by means of Girsanov's theorem.
\begin{lemma}        \label{lem:Krylov-Gyongy-estimate}
Let $p\in(1,\infty]$. Under \cref{asn:multiplicative-basic}, there exists $N$ depending only on $p, d, \lambda, \| \sigma\|_{C_2}$ and $\|b \|_{\mathbb{B}}$ such that for all $G \in L_p(\R^d)$ and $t\in(0,1]$
\begin{equs}
|\E G( X^n_t)|\leq N \|G\|_{L_p(\R^d)} t^{-d/(2p)}.
\end{equs}
\end{lemma}
From \cref{lem:Krylov-Gyongy-estimate}, using Khas'minskii's argument \cite{Khas} one can get estimates for exponential moments. For the adaptation of Khas'minskii's argument for the process $X^n$ we refer to \cite[Lemma~2.3]{Bao2020} or alternatively to \cite[Lemma~5.14]{le2021taming}.
\begin{lemma}        \label{lem:kasminskii}
Let  $q>(d+2)/2$.  There exist $\beta_q, \gamma_q \in (0, \infty)$ such that  for all $\mu>0$  
\begin{equs}
\E \exp \big( \mu \int_0^1 |f(X_s)| + |f(X^n_s)| \, ds \big) \leq
\exp\big(\beta_q( 1+ ( \mu \| f\|_{L_q((0,1) \times \R^d)})^{\gamma_q}) \big).
\end{equs}
Moreover $\lim_{q \to \infty} \gamma_q=1$, and there exists $\beta= \beta(  \lambda, \|\sigma\|_{C^2}, \|b\|_{\mathbb{B}}, d) \in \R$ such that $\lim_{q \to \infty}\beta_q= \beta$. 
\end{lemma}

\begin{proof}[Proof of \cref{thm:additive}]
We apply Lemma \ref{lem:kindofregularisation} with $X$ being the solution of the main SDE \eqref{eq:main} (and so $\bar X=X$, $\hat X=0$) and $Y=X^n$ being the solution of the approximate equation \eqref{eq:main-EM} (and so $\bar Y_t=X_{\kappa_n(t)}^n$, $\hat Y_t=X_t^n-X_{\kappa_n(t)}^n$).
Our task is therefore to choose the parameters $K, m,$ so that the right-hand side of \eqref{eq:reg-new} can be bounded by
$N n^{-p((1+\alpha)/2-\eps)}$. 
Note that in the case of \cref{thm:additive} $\bone_{\nabla\sigma\neq 0}=0$, eliminating one term from the right-hand side of \eqref{def.cR}.

It will be convenient to introduce two further parameters $\mu,q\in(1,\infty)$.
By Markov's inequality, we have, 
\begin{equs}
\mathbf{P} (A_1 \geq \frac{m}{2N} ) \leq \exp (- \mu m/ 2N) \E \exp (\mu  A_1).
\end{equs}
The maximal inequality \eqref{eq:H-L-new} and
Lemma \ref{lem:kasminskii} imply that 
\begin{equs}
\E \exp (\mu  A_1) \leq \exp \big( \beta_q(1+(\mu  \| |\nabla^2 u|^p \|_{L_q((0,1) \times \R^d)})^ {\gamma_q}) \big).
\end{equs}
By Lemma \ref{lem:PDE-estimates} we have 
\begin{equ}
\|u\|_{W^{1,2}_r}\leq N(r)\|b_K\|_{L_r(\R^d)}\leq K^{d/r},
\end{equ}
and therefore, for any sufficiently large $q$,
\begin{equs}\label{eq:PDE-K dependence}
\| |\nabla^2 u|^p \|_{L_q((0,1) \times \R^d)} \leq N(q) K^{d/q},\qquad \|\nabla u\|_{\bB([0,1],C^{1-\eps}(\R^d))}\leq N(q) K^{d/q},
\end{equs}
where the second inequality follows from Sobolev embedding.
Finally, notice that the last terms in \eqref{def.cR} are precisely the ones which were estimated in Corollary \ref{cor:Girsanov-additive}, with the choice $f=b$ and $g=1$ or $g=\nabla u$. 
Consequently, we obtain from \eqref{eq:reg-new} of \cref{lem:kindofregularisation} that
 \begin{equs}
\E  \sup_{ t \in [0,1]} |X_t-X^n_t|^p 
& \leq   N \exp (- \mu m/ 2N)   \exp \Big( (\mu \beta_q N(q) K^{d/q})^ {\gamma_q} \Big)
\\
&\qquad+ N^m \exp( -  K)   + N^m N(q) n^{-p((1+\alpha)/2-\eps)} K^{d/q}  
.
\end{equs}
Choose $K=p\ln n$ and $q$ large enough so that $d\gamma_q/q\leq 1/2$. Then, the above bound implies
\begin{equ}
\E  \sup_{ t \in [0,1]} |X_t-X^n_t|^p \leq N\exp(-\mu m/2N)\exp(\mu N(\ln n)^{1/2})+N^m n^{-p((1+\alpha)/2-2\eps)}.
\end{equ}
Now choose $m = \lfloor \frac{p\eps \ln n}{ \ln N }\rfloor$, so that $N^m\leq n^{p\eps}$. Then by choosing $\mu$ sufficiently large one can achieve $\exp (-\mu m/2N)\leq n^{-\gamma}$ for any exponent $\gamma$, which then yields the required bound
\begin{equ}
\E  \sup_{ t \in [0,1]} |X_t-X^n_t|^p \leq N n^{-p((1+\alpha)/2-3\eps)},
\end{equ} 
completing the proof.
\end{proof}

\begin{proof}[Proof of \cref{thm:multiplicative}]
Similarly to the previous proof, we apply Lemma \ref{lem:kindofregularisation}. This time $\bone_{\nabla\sigma\neq 0}=1$, so we have one more term to bound. However, we are aiming only at a bound of order $N n^{-p(1/2-\eps)}$.
Let $q\in(1,\infty)$ and $v\in L_q([0,1]\times\R^d)$. Then
from Lemma \ref{lem:Krylov-Gyongy-estimate} we have
\begin{equs}
\E\int_0^1|v(s,X^n_s)|\,ds&\leq N(q)
\int_0^1 \|v(s,\cdot)\|_{L_q(\R^d)}s^{-\frac d{2q}}\,ds
\\
&\leq N(q)\|v\|_{L_q([0,1]\times\R^d)}\Big(\int_0^1 s^{-\frac d{2(q-1)}}\,ds\Big)^{\frac{q-1}q}.\label{eq:385}
\end{equs}
For every $q>d/2+1$ the integral is finite. Similarly, we get, for the same range of $q$, 
\begin{equ}
\E\int_{1/n}^1|v(s,X^n_{\kappa_n(s)})|\,ds\leq N(q)\|v\|_{L_q([0,1]\times\R^d)}.\label{eq:282}
\end{equ}
Therefore, recalling that in the context of \eqref{eq:reg-new} we have $\hat X=0$, $\hat Y_t=X_t^n-X_{\kappa_n(t)}^n$, we can write
\begin{equs}
\sup_{t\in[0,1]}&\big(\E|X_t^n-X_{\kappa_n(t)}^n|^{2p}\big)^{1/2}\big(1+\E\int_0^1|\nabla^2 u(s,X^n_s)|^{2p}\,ds\big)^{1/2}
\\
&\leq N(q)n^{-p/2}\big(1+\big\||\nabla^2 u|^{2p}\big\|_{L_q([0,1]\times\R^d)}\big)^{1/2}
\\
&\leq N(q)n^{-p/2}K^{d/(2q)},\label{eq:481}
\end{equs}
where we used \eqref{eq:PDE-K dependence} in the last step. 
Concerning the remaining terms
\begin{equs}
\cR_1:&=\E\sup_{t\in[0,1]}\Big|\int_0^t \big(b(X^n_s)-b(X^n_{\kappa_n(s)})\big)\,ds\Big|^p,
\\
\cR_2:&=\E\sup_{t\in[0,1]}\Big|\int_0^t \big(b(X^n_s)-b(X^n_{\kappa_n(s)})\big)\nabla u(s,X^n_s)\,ds\Big|^p,
\end{equs}
we can treat $\cR_1$ just as before, using the results of Section \ref{sec:quad}, more precisely Corollary \ref{cor:Girsanov-multiplicative} with $f=b$. This yields the bound $\cR_1\leq N n^{-p(1/2-\eps)}$. For $\cR_2$ we argue slightly differently, since, unlike in Corollary \ref{cor:Girsanov-additive}, there is no `weight' function $g$ in Corollary \ref{cor:Girsanov-multiplicative}
(although it would not be too difficult to include, but we choose to have at least one of the integral estimates free of the tedium of weights). 
Instead we write
\begin{equs}
\cR_2\leq N\big(\cR_3+\cR_4+\cR_5\big):&=N\Big(\E\sup_{t\in[0,1]}\Big|\int_0^t \big((b\nabla u)(s,X^n_s)-(b\nabla u)(s,X^n_{\kappa_n(s)})\big)\,ds\Big|^p
\\
&\qquad+\E\sup_{t\in[1/n,1]}\Big|\int_{1/n}^t \big(\nabla u(s,X^n_s)-\nabla u(s,X^n_{\kappa_n(s)})\big)b(X^n_{\kappa_n(s)})\,ds\Big|^p
\\
&
\qquad+\E\sup_{t\in[0,1/n]}\Big|\int_0^{t} \big(\nabla u(s,X^n_s)-\nabla u(s,X^n_{\kappa_n(s)})\big)b(X^n_{\kappa_n(s)})\,ds\Big|^p
\Big).
\end{equs}
The term $\cR_3$ falls within the scope of Corollary \ref{cor:Girsanov-multiplicative} with $f=b\nabla u$, yielding the bound $\cR_3\leq N n^{-p(1/2+\eps)}$. The bound $\cR_5\leq N n^{-p}$ is trivial. It remains to bound $\cR_4$. Since for $s\geq 1/n$, $X^n_s$ and $X^n_{\kappa_n(s)}$ both have densities, we can apply \eqref{eq:maximal-new} to get
\begin{equs}
\cR_4&\leq N \E\int_{1/n}^1 \big|\nabla u(s,X^n_s)-\nabla u(s,X^n_{\kappa_n(s)})\big)|^p\,ds
\\
&\leq   N \E\int_{1/n}^1|X^n_s-X^n_{\kappa_n(s)}|^p\big(|\cM|\nabla^2 u(s, X^n_s)|+|\cM|\nabla^2 u(s, X^n_{\kappa_n(s)})|\big)^p\,ds
\\
&\leq  N \big(\sup_{s \in [0,1]}\E |X^n_s-X^n_{\kappa_n(s)}|^{2p}\big)^{1/2} \Big(  \int_{1/n}^1\E \big(|\cM|\nabla^2 u(s, X^n_s)|+|\cM|\nabla^2 u(s, X^n_{\kappa_n(s)})|\big)^{2p} \,ds\Big)^{1/2}
\\
&\leq Nn^{-p/2}\Big(  \int_{1/n}^1\E \big(|\cM|\nabla^2 u(s, X^n_s)|+|\cM|\nabla^2 u(s, X^n_{\kappa_n(s)})|\big)^{2p} \,ds\Big)^{1/2}.
\end{equs}
By \eqref{eq:H-L-new}, \eqref{eq:PDE-K dependence}, \eqref{eq:385}, and \eqref{eq:282},  the bound $\cR_4\leq N(q)n^{-p/2}K^{1/2q}$ follows.
Combining all of the above, \eqref{eq:reg-new} then implies
\begin{equ}
\E\sup_{t\in[0,1]}|X_t-X^n_t|^p\leq N\bP(A_1\geq \frac{m}{2N})+N^m\big(e^{-K}+N(q)n^{-p(1/2-\eps)}K^{d/q}\big).
\end{equ}
From here it remains to tune the parameter $K,q,m$, which is done similarly as in the previous proof, yielding the bound \eqref{eq:main-bound-multiplicative}.
\end{proof}

%
%

\textbf{Acknowledgments}
 The authors would like to thank the referees for their especially careful reading and many suggestions. The third author was supported by Alexander von Humboldt Research Fellowship. 

\bibliographystyle{Martin} 
\bibliography{Bounded_Drift}


\end{document}